\documentclass{amsart}
\usepackage{amsthm,amssymb}
\usepackage{graphicx}
\usepackage{pgf,tikz}

\newtheorem{theorem}{Theorem}[section]
\newtheorem{corollary}[theorem]{Corollary}
\newtheorem{lemma}[theorem]{Lemma}
\newtheorem{proposition}[theorem]{Proposition}

\theoremstyle{definition}
\newtheorem{definition}[theorem]{Definition}
\newtheorem{remark}[theorem]{Remark}
\newtheorem{example}[theorem]{Example}
\newtheorem{notation}[theorem]{Notation}

\newcommand{\F}{{\mathbb{F}}}
\newcommand{\X}{{\mathbb{X}}}
\newcommand{\Y}{{\mathbb{Y}}}
\newcommand{\V}{{\mathbb{V}}}
\newcommand{\W}{{\mathbb{W}}}
\newcommand{\Z}{{\mathbb{Z}}}
\newcommand{\N}{{\mathbb{N}}}

\renewcommand{\P}{{\mathbb{P}}}

\newcommand{\popo}{\mathbb{P}^1\times\mathbb{P}^1}

\newcommand{\supp}{\operatorname{Supp}}

\newcommand{\HF}{\operatorname{HF}}
\newcommand{\im}{\operatorname{im}}

\newcommand{\svdots}{\raisebox{3pt}{\scalebox{0.75}{$\vdots$}}}
\newcommand{\sddots}{\raisebox{3pt}{\scalebox{0.75}{$\ddots$}}}

\begin{document}

\title{K\"{a}hler Differentials for Fat Point Schemes
in~$\popo$}

\author{Elena Guardo}
\address{Dipartimento di Matematica e Informatica\\
Viale A. Doria, 6 \\
95100 - Catania, Italy}
\email{guardo@dmi.unict.it}

\author{Martin Kreuzer}
\address{Fakult\"{a}t f\"{u}r Informatik und Mathematik \\
Universit\"{a}t Passau, D-94030 Passau, Germany}
\email{martin.kreuzer@uni-passau.de}

\author{Tran N. K. Linh}
\address{Department of Mathematics\\
Hue University's College of Education\\
34 Le Loi, Hue, Vietnam}
\email{tnkhanhlinh141@gmail.com}

\author{Le Ngoc Long}
\address{Fakult\"{a}t f\"{u}r Informatik und Mathematik \\
Universit\"{a}t Passau\\
D-94030 Passau, Germany,\ \
\textrm{and} Department of Mathematics \\
Hue University's College of Education\\
34 Le Loi, Hue, Vietnam}
\email{nglong16633@gmail.com}

\date{\today}

\begin{abstract}
Let $\X$ be a set of $K$-rational points in~$\popo$
over a field~$K$ of characteristic zero,
let $\Y$ be a fat point scheme supported at~$\X$,
and let $R_{\Y}$ be the bihomogeneus coordinate ring of~$\Y$.
In this paper we investigate the module of K\"{a}hler
differentials $\Omega^1_{R_{\Y}/K}$.
We describe this bigraded $R_\Y$-module
explicitly via a homogeneous short exact sequence and
compute its Hilbert function in a number of special cases,
in particular when the support~$\X$ is a complete intersection
or an almost complete intersection in $\popo$.
Moreover, we introduce a K\"ahler different for~$\Y$
and use it to characterize ACM reduced schemes
in~$\popo$ having the Cayley-Bacharach property.
\end{abstract}

\maketitle

%
%

\section{Introduction}

It is a well-known fact that, over a field~$K$ of characteristic zero, 
the solutions of the Hermite interpolation problem for polynomials 
in two indeterminates correspond to the elements of the vanishing ideal 
of a fat point scheme in the plane. Here a 0-dimensional scheme~$\mathbb{Y}$
in $\mathbb{P}^2$ is called a {\it fat point scheme}\/ if its vanishing 
ideal~$I_{\mathbb{Y}}$ is of the form 
$$
I_{\Y} \; = \; I_{P_1}^{m_1} \cap \cdots \cap I_{P_r}^{m_r}
$$
with $m_i\in\mathbb{N}_+$ and points $P_i$ in $\mathbb{P}^2$.
The Hilbert function of $K[X_0,X_1,X_2]/I_{\Y}$ measures the number
of linearly independent solutions of the interpolation problem.
This is one of the reasons why Hilbert functions of fat point schemes
in~$\mathbb{P}^2$ have received a lot of attention. In spite of these
efforts, a complete classification of these Hilbert functions has been 
achieved only in some very special cases (such as the theorem of
J.\ Alexander and A.\ Hirschowitz for $m_1 = \cdots = m_r=2$, 
cf.~\cite{AH95}).

In another vein, M.\ Nagata conjectured in his 1959 solution of the
14th problem of Hilbert (cf.~\cite{Nag59}) that the initial degree~$d$ 
of~$I_{\Y}$ satisfies $d>(m_1+\cdots + m_r)/\sqrt{r}$ for $r\ge 9$
generically chosen points $P_1,\dots,P_r$
and proved it is some special cases. By blowing up the points
$P_1,\dots,P_r$, one can interpret the values of the Hilbert 
function of~$\Y$ as dimensions of linear series on the blown-up surface.
Thus the Nagata conjecture is related to the Harbourne-Hirschowitz
conjecture on the dimension of these linear systems and
can be extended to the Nagata-Biran conjecture for ample line bundles
on smooth algebraic surfaces (see for instance~\cite{SS04}).

Except for the projective plane, the simplest smooth algebraic
surface is $\popo$. It can be embedded into~$\mathbb{P}^3$ as a smooth
quadric surface, and it is also the simplest non-trivial example
of a multiprojective space $\mathbb{P}^{n_1} \times\cdots \times
\mathbb{P}^{n_k}$. Its bigraded coordinate ring $S=K[X_0,X_1,Y_0,Y_1]$
satisfies $\deg(X_0)=\deg(X_1)=(1,0)$ and $\deg(Y_0)=
\deg(Y_1)=(0,1)$, and the bigraded coordinate ring of a subscheme~$\Y$
of~$\popo$ is of the form $R_{\Y}=S/I_{\Y}$ with a bihomogeneous
ideal $I_{\Y}$ of~$S$. Hence the Hilbert function of~$\Y$ 
is given by an infinite matrix.

So far, not even the Hilbert functions of
reduced 0-dimensional schemes of~$\popo$ have been classified completely,
and the knowledge about Hilbert functions of fat point subschemes
of~$\popo$ is still rudimentary (see~\cite{GV15}, Ch.~6 and~\cite{VT05}).
This is unfortunate, since these Hilbert functions have many applications, e.g., 
in the study of secant varieties (see~\cite{CGG05}), complexity theory
(see~\cite{BCS97}), graphical models (see~\cite{GHKM01}), 
and Bayesian networks (see~\cite{GSS05}). 

To advance the study of fat point schemes in~$\popo$ and their Hilbert
functions further, new tools are needed.
In~\cite{DK99}, G. de Dominicis and the second author
introduced some methods using algebraic differential forms
into the study of 0-dimensional subschemes of~$\mathbb{P}^n$.
More precisely, given a 0-dimensional subscheme~$\X$
of the projective $n$-space $\P^n$ over a field~$K$
of characteristic zero with homogeneous vanishing
ideal~$I_\X$ in $R = K[X_0,\dots,X_n]$
and homogeneous coordinate ring $R_\X = R/I_\X$, let~$J$ be the
kernel of the multiplication map $\mu: R_{\X}\otimes_K R_{\X}
\rightarrow R_{\X}$.
Then the module of {\it K\"{a}hler differentials}
of $R_{\X}/K$ is the $R_{\X}$-module $\Omega^1_{R_{\X}/K} =
J/J^2$. The structure of this module can be described using
the canonical exact sequence
$$
0\rightarrow I^{(2)}_{\X}/I^2_{\X} \rightarrow I_{{\X}}/I^2_{\X}
\rightarrow R^{n+1}_{\X}(-1)\rightarrow
\Omega^1_{R_{\X}/K}\rightarrow 0
$$
which follows from~\cite[Prop.~4.13]{Ku86}.
For instance, if~$\X$ is the complete intersection of
hypersurfaces of degrees $d_1,\dots,d_n$
then it follows that the Hilbert function
of $\Omega^1_{R_{\X}/K}$ is given by
$\HF_{\Omega^1_{R_{\X}/K}}(i)
= (n + 1)\HF_{\X}(i-1)-\sum_{j=1}^n \HF_{\X}(i- d_j)$
for all $i\in \Z$ (see \cite[Prop.~4.3]{DK99}).
Later, in~\cite{KLL15}, these differential algebra techniques
were extended to fat point schemes in~$\mathbb{P}^n$.

In this paper we examine the natural question of whether these
differential algebraic methods can be applied to study
0-dimensional subschemes $\Y$ of~$\popo$. In particular, if
$\Y$ is a fat point scheme, if $S=K[X_0,X_1,Y_0,Y_1]$,
and if $R_{\Y}=S/I_{\Y}$ is the bihomogeneous coordinate
ring of~$\Y$, we show that the module of K\"ahler differentials
$\Omega^1_{R_{\Y}/K}$ contains a significant amount
of information about~$\Y$.

The paper is structured as follows. In  Section~2 we fix the
notation and recall a number of results about 0-dimensional
subschemes of~$\popo$ which we use later on. 
In particular, we recall
the definitions of arithmetically Cohen-Macaulay (ACM) 
subschemes, separators,
minimal separators, and the degree tuple of a set of
minimal separators for a fat point subscheme in~$\popo$.

In Section~3 we introduce the main object of the
study of this paper, namely the K\"{a}hler differential module
$\Omega^1_{R_{\Y}/K}$ for the bihomogeneous coordinate ring~$R_{\Y}$
of a fat point scheme~$\Y$ in~$\popo$. Based on the general theory
in~\cite{Ku86}, one can describe this module via generators and
relations. Our first main result is Theorem~\ref{generpropSec2.5}
which contains a more explicit presentation of the module
$\Omega^1_{R_{\Y}/K}$ via an exact sequence
$$
0\longrightarrow I_{\Y}/I_{\V} \longrightarrow
R_{\Y}^2(-1,0)\oplus R_{\Y}^2(0,-1)
\longrightarrow \Omega^1_{R_{\Y}/K}\longrightarrow 0
$$
where $\V$ is the fat point scheme obtained by increasing the
multiplicities of all points in~$\Y$ by one. This exact
sequence shows that one can compute the Hilbert function of the
K\"ahler differential module of $R_{\Y}/K$ from the Hilbert
functions of $\Y$ and $\V$. As Example~\ref{counterex}
shows, the exactness of this sequence depends on the hypothesis
that~$\Y$ is a fat point scheme and does not hold in general.

Section~4 contains a more detailed study of the Hilbert function
of $\Omega^1_{R_{\Y}/K}$ for a fat point scheme~$\Y$ in~$\popo$.
In particular, if~$\X$ is a set of
reduced points in~$\popo$ and~$\Y$ is a fat point scheme
$\Y = \{ (P_{ij}, m_{ij}) \mid P_{ij}\in \X\}$
supported at~$\X$, 
we can associate to~$\Y$ two tuples
$\alpha_{\Y}$ and $\beta_{\Y}$ as in~\cite{GV04}, \S~3.
Using these, we define two new tuples
$\hat{\alpha}_{\Y}$ and $\hat{\beta}_{\Y}$ and give formulas
for the Hilbert function of  $\Omega^1_{R_{\Y}/K}$
depending only on $\hat{\alpha}_{\Y}$ and $\hat{\beta}_{\Y}$
(see Prop.~\ref{PropSec3.3}). These formulas allow us to
describe the Hilbert function in large degrees explicitly
and to show that it is determined by
finitely many of its values.
If the corresponding fat point scheme
$\V = \{ (P_{ij}, m_{ij}+1) \mid P_{ij}\in \X\}$
is arithmetically Cohen-Macaulay (and hence,
if $m_{ij}>1$ for some $i,j$, also~$\Y$ is
arithmetically Cohen-Macaulay), we can describe the
Hilbert function of the K\"ahler differential module of~$\Y$
in terms of $\alpha_{\Y}$ and $\hat{\alpha}_{\Y}$
(see Prop.~\ref{ACMsupport}).

For a fat point scheme~$\Y$ whose support~$\X$
is a complete or an almost complete intersection, we give
even more detailed descriptions of the Hilbert function
of $\Omega^1_{R_{\Y}/K}$ in Section~5.
In particular, if~$\X$ is a reduced complete intersection
in~$\popo$ and~$\Y$ is an equimultiple
fat point scheme supported at~$\X$, we determine this Hilbert
function explicitly (see Prop.~\ref{equimultipleoverCI}).
Furthermore, if $m\geq 2$ and if $\Y =m\X$ is supported at a
complete intersection~$\X$, we show that the Hilbert function of
the K\"ahler differential module $\Omega^1_{R_{\Y_{ij}}/K}$ of
$\Y_{ij}=\sum_{(k,l) \neq (i,j)}mP_{kl}+(m-1)P_{ij}$
does not depend on the choice of $(i,j)$ (see Prop.~\ref{indip1}).
In the case $m=1$, the same result holds if $\X$ is a complete
intersection of type~$(h,h)$ (see Prop.~\ref{indip2}).
Of course, this independence on $(i,j)$ is reminiscent of the
Cayley-Bacharach Property (CBP) in~$\popo$ which we study in the
last section.

But before, in Section~6, we look at the K\"{a}hler different
of a fat point scheme~$\Y$ in~$\popo$.
For this we restrict our attention to arithmetically
Cohen-Macaulay fat
point schemes. Thus we may assume that
$x_0,y_0$ give rise to a regular sequence in~$R_\Y$.
Then the initial Fitting ideal
$\vartheta_\Y=F_0(\Omega^1_{R_\Y/K[x_0,y_0]})$
is also called the {\it K\"{a}hler
different} of~$\Y$ w.r.t.~$\{x_0,y_0\}$.
In Prop.~\ref{propSec5.12} we provide some basic properties
of the Hilbert function of~$\vartheta_{\Y}$ and show that
$\HF_{\vartheta_\Y} = 0$ if and only if~$\Y$ contains
no reduced point.
For an ACM set of reduced points, even more properties
of the Hilbert function of~$\vartheta_{\Y}$ are shown in
Prop.~\ref{HFofThetaForPoints}.

Finally, in the last section, we define the
Cayley-Bacharach property for a set~$\X$ of reduced
points in~$\popo$. If~$\X$ is ACM,
we characterize this property using the K\"ahler different, and
we write down the Hilbert function of~$\vartheta_{\X}$.

Unless explicitly mentioned otherwise, we use the
definitions and notation regarding K\"ahler differential modules,
K\"ahler differents, arithmetically Cohen-Ma\-cau\-lay
sets of points in $\popo$ and their properties found
in the books \cite{GV15, KR00, KR05, Ku86}.
All examples were computed using the
computer algebra system ApCoCoA (see \cite{ApCoCoA}).

%
%

\bigbreak
\section{Zero-Dimensional Schemes in $\popo$}

Let $K$ be a field of characteristic zero, and let
$S= K[X_0,X_1,Y_0,Y_1]$ be the bigraded polynomial ring
equipped with the $\Z^2$-grading defined by
$\deg X_0=\deg X_1=(1,0)$ and
$\deg Y_0=\deg Y_1=(0,1)$.
For $(i,j)\in \Z^2$, we let $S_{i,j}$
be the homogeneous component of degree $(i,j)$ of~$S$, i.e.,
the $K$-vector space with basis
$\{X_0^{\alpha_0}X_1^{\alpha_1}Y_0^{\beta_0}Y_1^{\beta_1}
\mid \alpha_0+\alpha_1=i,\, \beta_0+\beta_1=j,\,
\alpha_k,\beta_\ell\in\N\}$.

Note that $0\in S_{i,j}$ for all $i,j$,
and all elements of~$K$ have degree $(0,0)$.
Given two tuples $(i_1,j_1), (i_2,j_2) \in \Z^2$
we write $(i_1,j_1) \preceq (i_2,j_2)$ if $i_1\le i_2$
and $j_1\le j_2$.

Given an ideal $I \subseteq S$, we set $I_{i,j} : = I \cap
S_{i,j}$ for all $(i,j) \in \Z^2$.  Clearly, $I_{i,j}$ is a
$K$-vector subspace of~$S_{i,j}$ and $I \supseteq \bigoplus_{(i,j)
\in \Z^2} I_{i,j}$.  The ideal $I$ is called {\it bihomogeneous}
if $I = \bigoplus_{(i,j) \in \Z^2} I_{i,j}$. If~$I$ is a
bihomogeneous ideal of~$S$ then the quotient ring $S/I$
also inherits the structure of a bigraded ring via
$(S/I)_{i,j}:= S_{i,j}/I_{i,j}$ for all $(i,j) \in \Z^2$.

A finitely generated $S$-module $M$ is a {\it bigraded $S$-module}
if it has a direct sum decomposition
\[M = \bigoplus_{(i,j) \in \Z^2} M_{i,j}\]
with the property that $S_{i,j}M_{k,l} \subseteq M_{i+k,j+l}$
for all $(i,j),(k,l) \in \Z^2$.  For a bihomogeneous ideal~$I$
of~$S$, both~$I$ and $S/I$ are examples of bigraded $S$-modules.
Another example is the polynomial ring~$S$ equipped
with a {\it shifted grading}.
Specifically, for $(a,b) \in \Z^2$,
we let $S(-a,-b)_{i,j} := S_{i-a,j-b}$.
Notice that $S_{i,j} = 0$ if $(i,j)\nsucceq (0,0)$.

\begin{definition}
Let $M$ be a finitely generated bigraded $S$-module.
The {\it Hilbert function} of~$M$ is the numerical function
$\HF_{M}: \Z^2 \rightarrow \N$ defined by
\[
\HF_{M}(i,j) := \dim_K M_{i,j}
\quad \mbox{for all $(i,j)\in \Z^2$}.
\]
In particular, for a bihomogeneous ideal $I$ of $S$,
the Hilbert function of $S/I$ satisfies
\[
\HF_{S/I}(i,j) := \dim_k (S/I)_{i,j}
= \dim_k S_{i,j} - \dim_k I_{i,j}
\quad \mbox{for all $(i,j)\in \Z^2$}.\]
\end{definition}

If~$M$ is a finitely generated bigraded $S$-module such that
$\HF_M(i,j) =0$ for $(i,j) \nsucceq (0,0)$, we write
the Hilbert function of~$M$ as an infinite matrix,
where the initial row and column are indexed by~0.

\begin{example}
In the polynomial ring $S$, the set
of terms $\{X_0^aX_1^bY_0^cY_1^d \mid a+b =i, c+d=j\}$
forms a basis of the $K$-vector space $S_{i,j}$ for all
$(i,j)\succeq (0,0)$. So, the Hilbert function of $S$ satisfies
$\HF_{S}(i,j) = 0$ for $(i,j)\nsucceq (0,0)$ and
$$
\HF_{S}(i,j)=\dim_K S_{i,j} = (i+1)(j+1)
$$
for all $(i,j)\succeq (0,0)$.
In this case, we can write
$$
\HF_S =
\left[ \begin{smallmatrix}
1&2&3&4 & \cdots\\
2&4&6&8 & \cdots\\
3&6&9&12 & \cdots\\
4&8&12&16 & \cdots\\
\svdots & \svdots&
\svdots & \svdots & \sddots\\
\end{smallmatrix}\right].
$$
Next let us consider the monomial ideal
$I = \langle X_1,Y_1\rangle$ in $S$.
We have $\dim_K I_{0,0} =0$, and
$\dim_K I_{i,j} =(i+1)(j+1)-1$ for
$(i,j)\in \N^2\setminus\{(0,0)\}$
since every term of degree $(i,j)$
is an element of~$I_{i,j}$, except for $X_0^iY_0^j$.
So, the Hilbert function of $S/I$ is given by
$$
\HF_{S/I} =
\left[ \begin{smallmatrix}
1&1&1&1 & \cdots\\
1&1&1&1 & \cdots\\
1&1&1&1 & \cdots\\
1&1&1&1 & \cdots\\
\svdots & \svdots & \svdots & \svdots & \sddots\\
\end{smallmatrix}\right].
$$
\end{example}

Recall that a point in $\popo$ is of the form
$$
P= [a_0:a_1] \times [b_0:b_1] \in \popo
$$
where $[a_0:a_1],\, [b_0:b_1] \in \P^1$.
Its vanishing ideal is the bihomogeneous prime ideal
$$
I_P = \langle a_1X_0-a_0X_1, b_1Y_0-b_0Y_1\rangle.
$$
Let $\pi_1 : \popo \rightarrow \P^1$ be the projection morphism
given by $P_1\times P_2 \mapsto P_1$ and let
$\pi_2 : \popo \rightarrow \P^1$ be
the projection morphism given by $P_1\times P_2 \mapsto P_2$.

Let $\X$ be a set of~$s$ distinct points in~$\popo$. The set
$\pi_1(\X) = \{Q_1,\dots,Q_r\}$ is the set of $r\le s$
distinct first components of the points of~$\X$.
Similarly, the set $\pi_2(\X) = \{R_1,\dots,R_t\}$ is the
set of $t\le s$ distinct second components.
For $i=1,...,r$, let $L_{Q_i}$ denote the $(1,0)$-form
(i.e., the linear form in $K[X_0,X_1]$) which
vanishes at all points of~$\popo$ whose first component is~$Q_i$.
Similarly, for $j=1,...,t$, let $L_{R_j}$
denote the $(0,1)$-form that vanishes at all
the points of $\popo$ with second component~$R_j$.
Furthermore, we let
$$
D_\X := \{\, (i,j) \,\mid\, P_{ij}=Q_i\times R_j\in \X \,\}.
$$

\begin{definition}
For $(i,j)\in D_\X$, let $m_{ij}$ be a positive integer,
and let $\wp_{ij}$ be the vanishing ideal of the point
$P_{ij}\in \X$. Let~$\Y$ be the subscheme of~$\popo$
defined by the saturated bihomogeneous ideal
$$
I_{\Y} = \bigcap_{(i,j)\in D_\X} \wp_{ij}^{m_{ij}}.
$$
\begin{itemize}
\item[(a)]
The scheme $\Y$ is called a {\it fat point scheme}
of~$\popo$ and written as
    $$
    \Y = \{(P_{ij}; m_{ij}) \mid (i,j)\in D_\X\}
    \quad \mbox{or} \quad
    \Y = \sum_{(i,j)\in D_\X} m_{ij} P_{ij}.
    $$

\item[(b)]
  The integer $m_{ij}$ is called the {\it multiplicity}
  of the point $P_{ij}$ in~$\Y$.

\item[(c)]
  If $m_{ij}=m$ for all $(i,j)\in D_\X$, we denote $\Y$
  also by $m\X$ and call it an {\it equimultiple}
  (or {\it homogeneous}) fat point scheme.
\end{itemize}
\end{definition}

The bihomogeneous coordinate ring of
a fat point scheme~$\Y$ in~$\popo$ is given
by $R_\Y := S/I_\Y$ and its Hilbert function
will be denoted by~$\HF_\Y$. Also, the support
of~$\Y$ is $\X =\supp(\Y)$.

\begin{example} \label{examS2}
Let $i\ge 0$, let $Q_i = R_i = [1:i] \in \P^1$, let
$P_{ij}$ denote the point $Q_i\times R_j$ in $\popo$,
and let
$\X=\{P_{11},P_{12},P_{23},P_{31},P_{32}\}\subseteq\popo$.
The point $P_{ij}$ can be viewed as the intersection of
the horizontal ruling defined by $L_{Q_i} = X_1-iX_0$
and the vertical ruling defined by $L_{R_j} = Y_1-jY_0$.
Using the lines of the two rulings on $\popo$ passing through
the points of~$\X$, we can sketch the set~$\X$ as in the 
following figure (see \cite[Sec.~3.2]{GV15}).
\begin{center}
\begin{tikzpicture}[scale=.6, transform shape]
\filldraw[black] (1,3) circle (3pt);
\filldraw[black] (2,3) circle (3pt);
\filldraw[black] (3,2) circle (3pt);
\filldraw[black] (1,1) circle (3pt);
\filldraw[black] (2,1) circle (3pt);
\node (a1) at (0,3) {$L_{Q_1}$};
\node (a2) at (0,2) {$L_{Q_2}$};
\node (a3) at (0,1) {$L_{Q_3}$};
\node (a4) at (1,4) {$L_{R_1}$};
\node (a5) at (2,4) {$L_{R_2}$};
\node (a6) at (3,4) {$L_{R_3}$};
\draw[-] (3,0.5) -- (3, 3.5);
\draw[-] (2,0.5) -- (2, 3.5);
\draw[-] (1,0.5) -- (1, 3.5);
\draw[-] (0.5,3) -- (3.5,3);
\draw[-] (0.5,2) -- (3.5,2);
\draw[-] (0.5,1) -- (3.5,1);
\end{tikzpicture}
\end{center}
Let $\Y$ and $\V$ be the two fat point schemes
supported at~$\X$
$$
\Y = 2P_{11}+2P_{12}+P_{23}+P_{31}+2P_{32}
~\mbox{and}~
\V = 3P_{11}+3P_{12}+2P_{23}+2P_{31}+3P_{32}.
$$
Then, using ApCoCoA, we find that the Hilbert functions
of $\Y$ and of $\V$ are
$$
\HF_{\Y} =
\left[\begin{smallmatrix}
1 & 2 & 3 & 4 & 5 & 5 & \dots \\
2 & 4 & 6 & 8 & 8 & 8 & \dots \\
3 & 6 & 9 & 10 & 10 & 10 & \dots \\
4 & 8 & 10 & 11 & 11 & 11 & \dots \\
5 & 8 & 10 & 11 & 11 & 11 & \dots \\
5 & 8 & 10 & 11 & 11 & 11 & \dots \\
\svdots & \svdots & \svdots & \svdots &
\svdots & \svdots & \sddots
\end{smallmatrix}\right]
\hbox{\ and\ }
\HF_{\V} =
\left[\begin{smallmatrix}
1 & 2 & 3 & 4 & 5 & 6 & 7 & 8 & 8 & \dots \\
2 & 4 & 6 & 8 & 10 & 12 & 14 & 14 & 14 & \dots \\
3 & 6 & 9 & 12 & 15 & 18 & 18 & 18 & 18 & \dots \\
4 & 8 & 12 & 16 & 20 & 21 & 21 & 21 & 21 & \dots \\
5 & 10 & 15 & 20 & 22 & 23 & 23 & 23 & 23 & \dots \\
6 & 12 & 18 & 21 & 23 & 24 & 24 & 24 & 24 & \dots \\
7 & 14 & 18 & 21 & 23 & 24 & 24 & 24 & 24 & \dots \\
8 & 14 & 18 & 21 & 23 & 24 & 24 & 24 & 24 & \dots \\
8 & 14 & 18 & 21 & 23 & 24 & 24 & 24 & 24 & \dots \\
\svdots & \svdots & \svdots &\svdots & \svdots &
\svdots & \svdots & \svdots & \svdots & \sddots
\end{smallmatrix}\right]
$$
\end{example}

Now we introduce a special class of fat point schemes
in~$\popo$.

\begin{definition}
A fat point scheme $\Y$ of $\popo$ is called
{\it arithmetically Cohen-Macaulay} (or ACM for short)
if its bihomogeneous coordinate ring $R_{\Y}$
is a Cohen-Macaulay ring.
\end{definition}

In the case that $\Y$ is an ACM fat point scheme
in~$\popo$, we can find a regular sequence of length two
in~$R_\Y$ of the following form
(see \cite[Thm.~4.4]{GV04}).

\begin{theorem}\label{thmSec1.4}
Let $\Y$ be an ACM fat point scheme in~$\popo$.
Then there exist bihomogeneous polynomials $L_1,L_2 \in S$ with
$\deg(L_1)=(1,0)$ and $\deg(L_2)=(0,1)$
such that $L_1,\, L_2$ is a regular
sequence for~$R_{\Y}$.
\end{theorem}

Given a set of distinct points $\X$ in $\popo$,
and any point $P \in \X$, we sometimes want to
compare the properties of $\X$ with those
of~$\X \setminus \{P\}$.
A separator gives us a tool to compare and
contrast these two sets of points.

\begin{definition}\label{separator}
Let $\X$ be a set of distinct points in~$\popo$
and let $(i,j) \in D_\X$.
A bihomogeneous element $F\in S$ is called
a {\it separator} \index{separator} of~$P_{ij}$
if $F(P_{ij})\neq 0$, but $F(P_{kl})=0$ for
all $(k,l)\in D_\X\setminus\{(i,j)\}$.
We call $F$ a {\it minimal separator} of $P_{ij}$ if
there does not exist a separator $G$ of $P_{ij}$ with
$\deg G \prec \deg F$. In this case we also call
$\overline{F}\in R_\X$ a minimal separator of~$P_{ij}$.
\end{definition}

\begin{remark}
Note that if $F\in S$ is a separator of a point $P_{ij}$
of a set of distinct points $\X$ in $\popo$,
then the ideal $I_\X + \langle F\rangle$ is
the vanishing ideal of the set~$\X\setminus\{P_{ij}\}$.
\end{remark}

More generally, let $\Y = \sum_{(i,j)\in D_\X} m_{ij}P_{ij}$
be a fat point scheme in~$\popo$ (with support $\X =\supp(\Y)$),
and let $\wp_{ij}$ be the vanishing ideal of $P_{ij}$
for every $(i,j)\in D_\X$.

\begin{definition}
Let $(i,j)\in D_{\X}$, and let
$G\in S$ be a bihomogeneous element such that
$G \in \wp_{ij}^{m_{ij}-1} \setminus \wp_{ij}^{m_{ij}}$
and $G \in \wp_{kl}^{m_{kl}}$ for all
$(k,l)\in D_\X\setminus\{(i,j)\}$.
Then $G$ is called a {\it separator} of~$P_{ij}$ in~$\Y$.
\end{definition}

In this setting the preceding remark generalizes as follows.

\begin{proposition}
Let $\Y$ be a fat point subscheme of~$\popo$ as above,
let $(i,j)\in D_{\X}$, and let $G\in S$ be a separator
of~$P_{ij}$.
Then the ideal $I_\Y + \langle G\rangle$ defines
a subscheme of~$\Y$ of degree $\deg(\Y)-1$.
\end{proposition}

\begin{proof}
It suffices show that
$I_\Y : \langle G\rangle = \wp_{ij}$
and to apply the short exact sequence
$$
0\longrightarrow S/(I_\Y : \langle G\rangle)(-\deg(G))
\stackrel{\times G}{\longrightarrow} S/I_\Y
\longrightarrow  S/(I_\Y + \langle G\rangle)
\longrightarrow 0.
$$
Clearly, we have $\wp_{ij} \subseteq I_\Y:\langle G\rangle$.
For the other inclusion, let $H \in I_\Y:\langle G\rangle$
be a bihomogeneous element.
Then $GH \in I_\Y \subseteq \wp_{ij}^{m_{ij}}$.
Using a change of coordinates, we may assume that
$P_{ij} = [1:0]\times[1:0]$, and so
$\wp_{ij}=\langle X_1,Y_1\rangle$.
We write $H = aX_0^kY_0^l+H'$, where $H'\in \wp_{ij}$
with $\deg(H')=\deg(H)=(k,l)$ and $a\in K$.
We get $GH = aX_0^kY_0^lG + H'G$,
and so $aX_0^kY_0^lG \in \wp_{ij}^{m_{ij}}$.
Since $G\notin \wp_{ij}^{m_{ij}}$, Macaulay's Basis Theorem
(cf. \cite[Thm.~1.5.7]{KR00}) implies $a = 0$,
and hence $H = H'\in\wp_{ij}$.
\end{proof}

Next we recall the notion of a minimal separator
of a fat point in a fat point scheme in~$\popo$.

\begin{definition}\label{fatseparator}
Let $\Y = \sum_{(i,j)\in D_\X} m_{ij}P_{ij}$ be a fat
point scheme in~$\popo$, let $(i,j)\in D_\X$, and let
$\Y'= \sum_{(k,l)\in D_\X\setminus\{(i,j)\}} m_{kl}P_{kl}
+ (m_{ij}-1)P_{ij}$ be the fat point scheme
obtained by reducing the multiplicity of $P_{ij}$ by one.
(If $m_{ij}= 1$, then the point $P_{ij}$ does
not appear in the support of~$\Y'$.)

\begin{enumerate}
\item[(a)] A set $\{F_1,\dots,F_t\}$ is called a set
  of {\it minimal separators of $P_{ij}$ in~$\Y$} if
  $I_{\Y'}/I_{\Y}=
  \langle \overline{F}_1,\dots,\overline{F}_t\rangle$
  and if there does not exist a set $\{G_1,\dots,G_u\}$
  with $u<t$ such that
  $I_{\Y'}/I_{\Y} =
  \langle \overline{G}_1,\dots, \overline{G}_u\rangle$.
  In this case we also say that
  $\{\overline{F}_1,\dots, \overline{F}_t\} \subseteq R_\Y$
  is a set of minimal separators of $P_{ij}$ in~$\Y$.

\item[(b)] The {\it degree tuple} of a minimal separators
  of $P_{ij}$ in~$\Y$ is the tuple
  $$
  \deg_\Y(P_{ij}) = (\deg(F_1),\dots, \deg(F_t))
  $$
  where $\{F_1,\dots,F_t\}$ is any set of minimal
  separators of~$P_{ij}$ in~$\Y$, relabeled such that
  $\deg(F_1) \le_{\texttt{Lex}} \cdots
  \le_{\texttt{Lex}} \deg(F_t)$.
\end{enumerate}
\end{definition}

Basic information about minimal separators of fat points
in~$\popo$ can be found in~\cite{GV12}.
In particular, we note that the definition of $\deg_\Y(P_{ij})$ 
does not depend on the choice of $(F_1,\dots, F_t)$
and that these numbers are related to the graded Betti numbers
of~$I_{\Y}$ as follows.

\begin{theorem}\label{pnmsepfrombetti}
Let $\Y,\Y'$ be fat point schemes in $\popo$ as in
Definition~\ref{fatseparator}, and let
$\deg_\Y(P_{ij}) = (d_1,\ldots,d_t)$.
Suppose that $\Y$ is ACM, so that the minimal $\N^2$-graded
free resolution of $S/I_\Y$ has the form
\[ 0 \rightarrow \F_2=
\bigoplus_{(k,l) \in \Sigma} S(-k,-l)
\rightarrow \F_1 \rightarrow S
\rightarrow S/I_{\Y} \rightarrow 0\]
with a finite set $\Sigma \subset\N^2$.
If $\Y'$ is ACM then we have $t=m_{ij}$ and
$d_k + (1,1) \in \Sigma$ for $k=1,\dots, m_{ij}$.

In particular, if $\Y$ is an ACM set of distinct points
in~$\popo$, then the subset $\Y'=\Y \setminus \{P_{ij}\}$
is ACM if and only if
$\deg_\Y(P_{ij}) + (1,1)\in \Sigma$.
\end{theorem}

Finally, we recall the following formula to compute the degree tuple
of a fat point in an ACM fat point scheme in~$\popo$
(see \cite[Thm.~3.4]{GV12}).

\begin{theorem} \label{degP1xP1}
Let $\Y =\sum_{(k,l)\in D_\X}m_{kl}P_{kl}$ be an ACM fat point
scheme in~${\P^1\times\P^1}$, let $(i,j)\in D_\X$.
For every $\ell \in\{ 0,\ldots,m_{ij} -1\}$, we set
$$
a_{\ell} =
{\textstyle\sum\limits_{(e,j)\in D_\X}} \max\{m_{ej} - \ell,0\}
\quad \mbox{and} \quad
b_{\ell} =
{\textstyle\sum\limits_{(i,p)\in D_\X}} \max\{m_{ip} - \ell,0\}.
$$
Then we have
\[
\deg_\Y(P_{ij}) =
((a_{m_{ij}-1}-1,b_0-1),(a_{m_{ij}-2}-1,b_1-1),
\dots, (a_{0}-1,b_{m_{ij}-1}-1)).
\]
\end{theorem}

%
%

\bigbreak
\section{A Presentation of the K\"ahler Differential Module}

In the following we let $\Y$ be a fat point scheme
in~$\popo$ supported at $\X$.
For $i=0,1$, we denote the image of $X_i$ (respectively, $Y_i$)
in $R_{\Y}$ by $x_i$ (respectively, $y_i$).

By \cite[Lemma~1.2]{GV04}, there exist bihomogeneous elements
$L_1\in S$ of degree $(1,0)$ and $L_2\in S$ of degree $(0,1)$
which are non-zerodivisors for~$R_{\Y}$.
After linear changes of coordinates in $\{X_0,X_1\}$
and in $\{Y_0,Y_1\}$, we may
assume that $L_1 = X_0$ and $L_2 =Y_0$, and hence that
$x_0,y_0$ are non-zerodivisors in~$R_{\Y}$.
In the bigraded algebra
$$
R_{\Y}\otimes_{K}R_{\Y} =
{\textstyle\bigoplus\limits_{(i,j)\in\N^2}}
({\textstyle \bigoplus\limits_{
\begin{subarray} \ i_1+i_2 =i \\ j_1+j_2=j\end{subarray}}}
(R_{\Y})_{i_1,j_1}\otimes (R_{\Y})_{i_2,j_2})
$$
we have the bihomogeneous ideal $J=\ker(\mu)$,
where $\mu:\ R_{\Y}\otimes_{K} R_{\Y}\rightarrow R_{\Y}$
is the bihomogeneous $R_{\Y}$-linear map given by
$\mu(f\otimes g)=fg$.
Notice that
$$
J = \langle x_i\otimes1-1\otimes x_i,
y_i\otimes1-1\otimes y_i \mid i=0,1\rangle.
$$

\begin{definition}
The bigraded $R_{\Y}$-module $\Omega^1_{R_{\Y}/K}= J/J^2$
is called the {\it module of K\"{a}hler differentials}
of~$R_{\Y}/K$. The bihomogeneous $K$-linear map
$d_{R_{\Y}/K}:R_{\Y}\rightarrow \Omega^1_{R_{\Y}/K}$
given by $f\mapsto f\otimes 1-1\otimes f+J^2$ is called
the {\it universal derivation} of $R_{\Y}/K$.
\end{definition}

We collect some useful properties of the module
of K\"{a}hler differentials.
For a proof of these properties, see
\cite[Props.~4.12 and 4.13]{Ku86}.

\begin{proposition}\label{propSec2.2}
\begin{enumerate}
\item[(a)] There is an isomorphism of bigraded
    $R_{\Y}$-modules
    $$
    \Omega^1_{R_{\Y}/K} \cong
    \Omega^1_{S/K}/(dI_{\Y}+ I_{\Y}\Omega^1_{S/K})
    $$
    where $\deg(dX_i)=(1,0)$ and $\deg(dY_i)=(0,1)$,
    and where
    $$
    \qquad \quad d I_{\Y} =
    \langle \tfrac{\partial F}{\partial X_0}dX_0 +
    \tfrac{\partial F}{\partial X_1}dX_1+
    \tfrac{\partial F}{\partial Y_0}dY_0
    +\tfrac{\partial F}{\partial Y_1}dY_1
    \mid F \in I_{\Y}\rangle.
    $$

\item[(b)] The elements $\{dx_0,dx_1,dy_0,dy_1\}$ form
  a bihomogeneous system of generators
  of the bigraded $R_{\Y}$-module $\Omega^1_{R_{\Y}/K}$.
\end{enumerate}
\end{proposition}

\begin{remark}\label{remSec2.3}
\begin{itemize}
\item[(a)] The bigraded $S$-module $\Omega^1_{S/K}$ has
  the representation
  $$
  \Omega^1_{S/K} =
   SdX_0\oplus SdX_1 \oplus SdY_0
   \oplus SdY_1 \cong S^2(-1,0)\oplus S^2(0,-1).
  $$

\item[(b)] If $I_1,I_2$ are bihomogeneous ideals of $S$, then
  $$
  I_1\Omega^1_{S/K} \cap I_2\Omega^1_{S/K}
  = (I_1\cap I_2)\Omega^1_{S/K}
  $$
  (see \cite[Ch.~3, \S 7, Thm.~7.4(i)]{Mat87}).
\end{itemize}
\end{remark}

For convenience, we introduce the following notion.

\begin{definition}
Let $\Y = \sum_{(i,j)\in D_\X} m_{ij}P_{ij}$
be a fat point scheme in~$\popo$.
The fat point scheme
$\V = \sum_{(i,j)\in D_\X} (m_{ij}+1)P_{ij}$
is called the {\it thickening} of~$\Y$.
\end{definition}

In analogy with~\cite[Thm.~1.7]{KLL15}, we have the
following presentation of the module of K\"{a}hler differentials
of $R_\Y/K$ when $\Y$ is a fat point scheme in $\popo$.

\begin{theorem} \label{generpropSec2.5}
Let $\Y = \sum_{(i,j)\in D_\X} m_{ij}P_{ij}$ be a fat point
scheme in~$\popo$, and let $\V$ be the thickening of~$\Y$.
Then the sequence of bigraded $R_{\Y}$-modules
$$
0\longrightarrow I_{\Y}/I_{\V} \longrightarrow
R_{\Y}^2(-1,0)\oplus R_{\Y}^2(0,-1)
\longrightarrow \Omega^1_{R_{\Y}/K}\longrightarrow 0
$$
is exact.
\end{theorem}

In the proof of this theorem we use the following lemma.

\begin{lemma}\label{lemforpropSec2.5}
In the setting of Theorem~\ref{generpropSec2.5}, let
$\varphi: I_{\Y}/I_{\V} \rightarrow
\Omega^1_{S/K}/I_{\Y}\Omega^1_{S/K}$
be the map given by
$\varphi(F+I_{\V})=dF+I_{\Y}\Omega^1_{S/K}$
for all $F\in I_{\Y}$.
Then the map $\varphi$ is well-defined, bihomogeneous
of degree $(0,0)$,
and $R_{\Y}$-linear.
\end{lemma}

\begin{proof}
For $(i,j)\in D_\X$ we let $\wp_{ij}$ be the associated ideal
of~$P_{ij}\in \X=\supp(\Y)$.
First we check that the map~$\varphi$ is well defined.
Let $F_1, F_2\in I_{\Y}$ be such that $F_1-F_2\in I_{\V}$.
Then, for every $P_{ij}=Q_i\times R_j \in \X$,
we have $F_1-F_2\in \wp_{ij}^{m_{ij}+1}$.
Since $\wp_{ij} = \langle L_{Q_i}, L_{R_j}\rangle$, we have
$\wp_{ij}^{m_{ij}+1} = \langle L_{Q_i}^{m_{ij}+1},
L_{Q_i}^{m_{ij}}L_{R_j},\dots, L_{R_j}^{m_{ij}+1}\rangle$.
So, we get
\[
dF_1-dF_2 \in d(\wp_{ij}^{m_{ij}+1}) \subset
\langle L_{Q_i}^{m_{ij}},L_{Q_i}^{m_{ij}-1}L_{R_j},
\dots, L_{R_j}^{m_{ij}}\rangle\Omega^1_{S/K}
= \wp_{ij}^{m_{ij}}\Omega^1_{S/K}.
\]
Hence we see that
\[
dF_1 - dF_2 \in \bigcap_{(i,j)\in D_\X}
\wp_{ij}^{m_{ij}}\Omega^1_{S/K}.
\]
Because $\Omega^1_{S/K}$ is a free $S$-module of rank~$4$,
Remark~\ref{remSec2.3} yields that
\[
dF_1 - dF_2 \in \bigcap_{(i,j)\in D_\X}
\wp_{ij}^{m_{ij}}\Omega^1_{S/K}
= I_{\Y}\Omega^1_{S/K}.
\]

Moreover, it is clearly true that the map $\varphi$ is
bihomogeneous of degree $(0,0)$.
Now we prove that the map~$\varphi$ is $R_{\Y}$-linear.
For $F_1, F_2 \in I_{\Y}$ and $G_1,G_2 \in S$, we have
$$
\begin{aligned}
\varphi(G_1F_1+G_2F_2+I_{\V})
&= d(G_1F_1+G_2F_2) + I_{\Y}\Omega^1_{S/K}\\
&= G_1dF_1+G_2dF_2 + I_{\Y}\Omega^1_{S/K}\\
&= (G_1+I_{\Y})\cdot\varphi(F_1+I_{\V})+(G_2+I_{\Y})
\cdot\varphi(F_2+I_{\V}).
\end{aligned}
$$
Hence the conclusion follows.
\end{proof}

\begin{proof}[Proof of Theorem~\ref{generpropSec2.5}]
Let $\varphi: I_{\Y}/I_{\V} \rightarrow
\Omega^1_{S/K}/I_{\Y}\Omega^1_{S/K}$
be the $R_{\Y}$-linear map given in Lemma~\ref{lemforpropSec2.5}.
We see that
$$
\im(\varphi)=\varphi(I_{\Y}/I_{\V}) =
(dI_{\Y}+I_{\Y}\Omega^1_{S/K})/ I_{\Y}\Omega^1_{S/K},
$$
and so
\[
(\Omega^1_{S/K}/I_{\Y}\Omega^1_{S/K}) / \im(\varphi)
\cong \Omega^1_{S/K}/(dI_{\Y}+I_{\Y}\Omega^1_{S/K})
\cong \Omega^1_{R_{\Y}/K}.
\]
In addition, we have
$$
R_{\Y}^2(-1,0)\oplus R_{\Y}^2(0,-1)
\cong \Omega^1_{S/K}/I_{\Y} \Omega^1_{S/K}.
$$
Hence it is enough to prove that the $R_{\Y}$-linear
map $\varphi$ is an injection.
To this end, let $F \in I_{\Y}\setminus I_{\V}$ be
a bihomogeneous form of degree $(k,l) \in \N^2$.
We need to prove $\varphi(F)=dF+ I_{\Y}\Omega^1_{S/K}\ne 0$
or $dF\notin I_{\Y}\Omega^1_{S/K}$.
Since $I_{\Y}\Omega^1_{S/K} = \bigcap_{(i,j)\in D_\X}
\wp_{ij}^{m_{ij}}\Omega^1_{S/K}$,
it suffices to show that
$dF \notin \wp_{ij}^{m_{ij}}\Omega^1_{S/K}$
for some $(i,j)\in D_\X$.

We have $I_{\Y} = \bigcap_{(i,j)\in D_\X} \wp_{ij}^{m_{ij}}$ and
$I_{\V} = \bigcap_{(i,j)\in D_\X} \wp_{ij}^{m_{ij}+1}$, and hence
it follows from $F \in I_{\Y}\setminus I_{\V}$ that there is
a point $P_{i_0,j_0} \in \X$ such that $F\in
\wp_{i_0j_0}^{m_{i_0j_0}}$,
but $F \notin \wp_{i_0j_0}^{m_{i_0j_0}+1}$.
Since $x_0,y_0$ are non-zerodivisors of~$R_{\Y}$,
w.l.o.g. we may assume that
\[
P_{i_0j_0} = P = [1:0]\times [1:0] \in \popo.
\]
Set $m = m_{i_0j_0}$ and write
$\wp_{i_0j_0} = \wp = \langle X_1,Y_1\rangle \subseteq S$.
Notice that
$\wp^{t} = \langle X_1^{i}Y_1^{j} \mid i+j = t \rangle$
for all $t\ge 1$.
Since $F \in \wp^{m} \setminus \wp^{m+1}$ and $\deg(F)=(k,l)$,
we have the following representation of $F$:
$$
F =
\sum_{\begin{subarray}{l}\
        i+j=m+1\\
       (i,j)\preceq(k,l)
      \end{subarray}}
       X_1^iY_1^j F_{i,j} +
\sum_{\begin{subarray}{l} \
       i+j=m\\(i,j)\preceq(k,l)
      \end{subarray}}
       X_1^iY_1^j G_{i,j}
$$
where $F_{i,j} \in S$ with $i+j = m+1$ and
$\deg(F_{i,j}) = (k-i,l-j)$,
and where $G_{i,j} \in K[X_0,Y_0]$ with $i+j=m$
and $\deg(G_{i,j})=(k-i,l-j)$.
We write $G_{i,j} = a_{i,j} X_0^{k-i}Y_0^{l-j}$ with
$a_{i,j}\in K$.
Since $F \notin \wp^{m+1}$, not all elements
$a_{i,j}$ are zero. Setting
$\widetilde{F}=\sum_{i+j=m+1,(i,j)\preceq(k,l)} X_1^iY_1^j F_{i,j}$,
we have $d\widetilde{F} \in \wp^{m}\Omega^1_{S/K}$ and
$$
\begin{aligned}
dF &= d\widetilde{F} +
d({\textstyle \sum\limits_{
\begin{subarray}{l}\
i+j=m\\(i,j)\preceq(k,l)
\end{subarray}}}
 X_1^iY_1^j G_{i,j}) \\
&=  d\widetilde{F} +
{\textstyle \sum\limits_{
\begin{subarray}{l}\
i+j=m\\(i,j)\preceq(k,l)
\end{subarray}}}
d(a_{i,j}X_1^iY_1^j X_0^{k-i}Y_0^{l-j}) \\
&=  \big(d\widetilde{F} +
{\textstyle \sum\limits_{
\begin{subarray}{l}\
i+j=m\\(i,j)\preceq(k,l)
\end{subarray}}}
a_{i,j}X_1^iY_1^j d(X_0^{k-i}Y_0^{l-j})\big) +
{\textstyle \sum\limits_{
\begin{subarray}{l}\
i+j=m\\(i,j)\preceq(k,l)
\end{subarray}}}
a_{i,j}X_0^{k-i}Y_0^{l-j} d(X_1^iY_1^j).
\end{aligned}
$$
Note that
\[
d\widetilde{F} +
{\textstyle \sum\limits_{
\begin{subarray}{l}\
i+j=m\\(i,j)\preceq(k,l)
\end{subarray}}}
a_{i,j}X_1^iY_1^j d(X_0^{k-i}Y_0^{l-j})
\in \wp^{m}\Omega^1_{S/K}.
\]
So, in order to prove $dF \notin \wp^{m}\Omega^1_{S/K}$,
it suffices to prove that
\[
w :=
{\textstyle \sum\limits_{
\begin{subarray}{l}\
i+j=m\\(i,j)\preceq(k,l)
\end{subarray}}}
a_{i,j}X_0^{k-i}Y_0^{l-j} d(X_1^iY_1^j)
\notin \wp^{m}\Omega^1_{S/K}.
\]
Observe that
$$
\begin{aligned}
w &=
{\textstyle \sum\limits_{
\begin{subarray}{l}\
i+j=m\\(i,j)\preceq(k,l)
\end{subarray}}}
a_{i,j}X_0^{k-i}Y_0^{l-j}
(i X_1^{i-1}Y_1^jdX_1 + jX_1^{i}Y_1^{j-1} dY_1) \\
&= \big(
{\textstyle \sum\limits_{
\begin{subarray}{l}\
i+j=m\\(i,j)\preceq(k,l)
\end{subarray}}}
ia_{i,j}X_0^{k-i}Y_0^{l-j}X_1^{i-1}Y_1^j\big)dX_1+
\big(
{\textstyle \sum\limits_{
\begin{subarray}{l}\
i+j=m\\(i,j)\preceq(k,l)
\end{subarray}}}
ja_{i,j}X_0^{k-i}Y_0^{l-j}X_1^{i}Y_1^{j-1}\big) dY_1.
\end{aligned}
$$
The ideal $\wp^{m}$ is a monomial ideal of $S$ generated by
the set $\{ X_1^{i}Y_1^{j} \mid i+j = m\}$.
Macaulay's Basis Theorem yields that for all
$(i,j)\in\N^2$ with $i+j = m$ and $(i,j)\preceq(k,l)$ we have
$X_0^{k-i}Y_0^{l-j}X_1^{i-1}Y_1^{j} \notin \wp^{m}$
if $i>0$ and $X_0^{k-i}Y_0^{l-j}X_1^{i}Y_1^{j-1} \notin \wp^{m}$
if $j>0$, and that any non-zero polynomial which has support
contained in the set of these terms does not belong to $\wp^m$.
Because not all elements $a_{i,j}$ are zero
and $i+j=m\ge 1$, it follows that
\[
\big(
{\textstyle \sum\limits_{
\begin{subarray}{l}\
i+j=m\\(i,j)\preceq(k,l)
\end{subarray}}}
ia_{i,j}X_0^{k-i}Y_0^{l-j}X_1^{i-1}Y_1^j\big)dX_1
\notin \wp^{m}\Omega^1_{S/K}
\]
or
\[
\big(
{\textstyle \sum\limits_{
\begin{subarray}{l}\
i+j=m\\(i,j)\preceq(k,l)
\end{subarray}}}
ja_{i,j}X_0^{k-i}Y_0^{l-j}X_1^{i}Y_1^{j-1}\big)dY_1
\notin \wp^{m}\Omega^1_{S/K}
\]
Note that $\Omega^1_{S/K}$ is free $S$-module with
basis $\{dX_0,dX_1,dY_0,dY_1\}$.
So, we get $w \notin \wp^m\Omega^1_{S/K}$.
Hence we have shown that
$dF \notin \wp^m\Omega^1_{S/K} =
\wp_{i_0j_0}^{m_{i_0j_0}}\Omega^1_{S/K}$.
This means $\varphi(F)\ne 0$ for any
$F \in I_{\Y}\setminus I_{\V}$.
Therefore $\varphi$ is an injection, as wanted.
\end{proof}

Using Theorem \ref{generpropSec2.5}, we get a relation
between the Hilbert function of $\Omega^1_{R_{\Y}/K}$
and of~$\Y$ and $\V$.

\begin{corollary}\label{corSec2.6}
Let $\Y = \sum_{(i,j)\in D_\X} m_{ij}P_{ij}$
be a fat point scheme in~$\popo$, and let
$\V$ be the thickening of~$\Y$. Then the Hilbert
function of~$\Omega^1_{R_{\Y}/K}$ satisfies
$$
\HF_{\Omega^1_{R_{\Y}/K}}(i,j)=2\HF_{\Y}(i-1,j)+
2\HF_{\Y}(i,j-1)+\HF_{\Y}(i,j)-\HF_{\V}(i,j)
$$
for all $(i,j)\in\Z^2$.
\end{corollary}

In the following example, we show that the exact sequence 
in Theorem~\ref{generpropSec2.5} does, in general, not hold true 
if $\Y$ is a 0-dimensional subscheme of~$\popo$ which is not
a fat point scheme.

\begin{example}\label{counterex}
Let $\Y$ be the 0-dimensional subscheme of~$\popo$ defined by the 
bihomogeneous ideal $I_\Y = \langle X_1^2,Y_0-Y_1\rangle$ of~$S$.
We have ${\rm Supp}(\Y)=\{[1:0]\times[1:1]\}$ and 
$\Y$ is neither a reduced scheme nor a fat point scheme in~$\popo$.
Let~$\V$ be the scheme defined by $(I_\Y)^2$. 
The Hilbert functions of $\Y$ and $\V$ are
$$
\HF_{\Y} =
\left[ \begin{smallmatrix}
1 & 1 & 1 & 1 & \dots \\
2 & 2 & 2 & 2 & \dots \\
2 & 2 & 2 & 2 & \dots \\
2 & 2 & 2 & 2 & \dots \\
\svdots & \svdots & \svdots & \svdots & \sddots
\end{smallmatrix}\right] \quad \mbox{and} \quad
\HF_{\V} =
\left[ \begin{smallmatrix}
1 & 2 & 2 & 2 & \dots \\
2 & 4 & 4 & 4 &  \dots \\
3 & 5 & 5 & 5 & \dots \\
4 & 6 & 6 & 6 & \dots \\
\svdots & \svdots & \svdots & \svdots & \sddots
\end{smallmatrix}\right].
$$
So, the Hilbert matrix $\HF$ computed by the formula 
$\HF(i,j) = 2\HF_{\Y}(i-1,j)+
2\HF_{\Y}(i,j-1)+\HF_{\Y}(i,j)-\HF_{\V}(i,j)$ is
$$
\HF =
\left[ \begin{smallmatrix}
0 & 1 & 1 & 1 & \dots \\
2 & 4 & 4 & 4 &  \dots \\
3 & 5 & 5 & 5 & \dots \\
2 & 4 & 4 & 4 & \dots \\
\svdots & \svdots & \svdots & \svdots & \sddots
\end{smallmatrix}\right].
$$
However, when we compute the Hilbert function 
of $\Omega^1_{R_{\Y}/K}$ by its definition, we get 
$$
\HF_{\Omega^1_{R_{\Y}/K}} =
\left[ \begin{smallmatrix}
0 & 1 & 1 & 1 & \dots \\
2 & 4 & 4 & 4 &  \dots \\
3 & 5 & 5 & 5 & \dots \\
3 & 5 & 5 & 5 & \dots \\
\svdots & \svdots & \svdots & \svdots & \sddots
\end{smallmatrix}\right].
$$
Thus this implies $\HF_{\Omega^1_{R_{\Y}/K}} \ne \HF$,
and hence we do not have the 
exact sequence as in Theorem~\ref{generpropSec2.5}
in this case.
\end{example}

Our next proposition collects some special values
of the Hilbert function of $\Omega^1_{R_{\Y}/K}$.

\begin{proposition}
Let $\Y = \sum_{(i,j)\in D_\X} m_{ij}P_{ij}$
be a fat point scheme in~$\popo$ with support~$\X =\supp(\Y)$.
The following statements hold true.
\begin{enumerate}
\item[(a)]  $\HF_{\Omega^1_{R_{\Y}/K}}(i,j) = 0$ if
    $(i,j) \nsucceq (0,0)$ or $(i,j)=(0,0)$.
  
\item[(b)] For $(i,j)\in \N^2$, if $\HF_{I_\Y}(i,j)=0$, then
    $$
    \HF_{\Omega^1_{R_{\Y}/K}}(i,j) = 4ij + 2i +2j.
    $$
    
\item[(c)] If $\HF_{\Y}(i_0-1,j_0) = \HF_{\Y}(i_0,j_0)
    =\HF_{\Y}(i_0,j_0-1)$, then for $(i,j)\succeq(i_0,j_0)$
    we have
    $$
    \HF_{\Omega^1_{R_{\Y}/K}}(i,j) = 5\HF_\Y(i,j)-\HF_\V(i,j)
    $$
    where $\V$ is the thickening of~$\Y$.
\end{enumerate}
\end{proposition}

\begin{proof}
Claim (a) follows from the representation of K\"{a}hler
differential module
$$
\Omega^1_{R_{\Y}/K}
= R_\Y dx_0 + R_\Y dx_1 + R_\Y dy_0 + R_\Y dy_1.
$$
Observe that $I_\V \subseteq I_\Y$ where
$\V = \sum_{(i,j)\in D_\X} (m_{ij}+1)P_{ij}$.
So, for $(i,j)\in \N^2$ such that $\HF_{I_\Y}(i,j)=0$
we have $\HF_{I_\V}(i,j)=0$ and
$\HF_\Y(i,j)= \HF_\V(i,j)= (i+1)(j+1)$.
Hence claim~(b) follows from Corollary~\ref{corSec2.6}.
Finally, claim~(c) is a consequence
of Corollary~\ref{corSec2.6}
and \cite[Prop.~1.3]{GV04}.
\end{proof}

Using the preceding results, we can calculate
the Hilbert function of the K\"ahler differential module
of a fat point scheme~$\Y$ on~$\popo$. Let us see an example.

\begin{example}\label{examS3.2}
Let $\X=\{P_{11},P_{12},P_{23},P_{31},P_{32}\}\subseteq\popo$
be a set of distinct points as considered in Example \ref{examS2}.
Let $\Y$ and $\V$ be the two fat point schemes
$$
\Y = 2P_{11}+2P_{12}+P_{23}+P_{31}+2P_{32}
~\mbox{and}~
\V = 3P_{11}+3P_{12}+2P_{23}+2P_{31}+3P_{32}.
$$

Based on the definition of~$\Omega^1_{R_{\Y}/K}$, or by applying
Corollary~\ref{corSec2.6}, we compute the
Hilbert function of~$\Omega^1_{R_{\Y}/K}$ and get
$$
\HF_{\Omega^1_{R_{\Y}/K}} =
\left[ \begin{smallmatrix}
0 & 2 & 4 & 6 & 8 & 9 & 8 & 7 & 7 & \dots \\
2 & 8 & 14 & 20 & 24 & 22 & 20 & 20 & 20 & \dots \\
4 & 14 & 24 & 32 & 31 & 28 & 28 & 28 & 28 & \dots \\
6 & 20 & 32 & 35 & 33 & 32 & 32 & 32 & 32 & \dots \\
8 & 24 & 31 & 33 & 33 & 32 & 32 & 32 & 32 & \dots \\
9 & 22 & 28 & 32 & 32 & 31 & 31 & 31 & 31 & \dots \\
8 & 20 & 28 & 32 & 32 & 31 & 31 & 31 & 31 & \dots \\
7 & 20 & 28 & 32 & 32 & 31 & 31 & 31 & 31 & \dots \\
7 & 20 & 28 & 32 & 32 & 31 & 31 & 31 & 31 & \dots \\
\svdots & \svdots & \svdots &\svdots & \svdots &
\svdots & \svdots & \svdots & \svdots & \sddots
\end{smallmatrix}\right]
$$
Let us explain some of these values.
Clearly, we have $\HF_{\Omega^1_{R_{\Y}/K}}(0,0)=0$.
For $(i,j)=(2,2)$, we have $(I_{\Y})_{2,2}=\langle0\rangle$.
This shows
$\HF_{\Omega^1_{R_{\Y}/K}}(2,2) = 24 = 4ij + 2i + 2j$.
For $(i,j)\succeq (4,4)$, we also see that
$\HF_{\Y}(3,4) = \HF_{\Y}(4,4) =\HF_{\Y}(4,3)$, and hence
$\HF_{\Omega^1_{R_{\Y}/K}}(i,j)
= 5 \HF_{\Y}(i,j) - \HF_{\V}(i,j)$.
In particular, for $(i,j) \succeq (5,5)$, we get
$\HF_{\Omega^1_{R_{\Y}/K}}(i,j) = 31
= 5 \HF_{\Y}(i,j) - \HF_{\V}(i,j)$.
\end{example}

%
%

\bigbreak
\section{The Hilbert Function of
the K\"{a}hler Differential Module}

Let us continue to use the setting of Section~2.
For a fat point scheme~$\Y$ in~$\popo$,
the authors of~\cite{GV04} associated to~$\Y$ two tuples $\alpha_{\Y}$
and $\beta_{\Y}$ in order to describe all but finitely many values
of the Hilbert function of~$R_{\Y}$.
In this section we show that the same tuples enable us
to describe all but a finite number of values of the
Hilbert function of the module of K\"{a}hler differentials
of~$\Y$.

Let $\pi_1 : \popo \rightarrow \P^1$ and
$\pi_2 : \popo \rightarrow \P^1$ be the projection morphisms
given by $P_1\times P_2 \mapsto P_1$ and
$P_1\times P_2 \mapsto P_2$, respectively.
Now we recall from \cite{GV04} the following notation
and definitions.

\begin{notation} \label{notationSec4.3}
Let $\Y = \sum_{(i,j)\in D_\X} m_{ij}P_{ij}$ be a fat
point scheme in~$\popo$ with support $\X =\supp(\Y)$,
and let $\pi_1(\X) =\{Q_1,\dots,Q_r\}$ and
$\pi_2(\X) = \{R_1,\dots,R_t\}$.

\begin{enumerate}
\item[(a)] For each $Q_i \in \pi_1(\X)$, we define
  $\Y_{1,Q_i} :=
  m_{ij_1}P_{ij_1}+ \cdots + m_{ij_{\nu_i}}P_{ij_{\nu_i}}$
  where $P_{ij_k} = Q_i\times R_{j_k}$ are those points of~$\X$
  whose first coordinate is $Q_i$. Set
  $$
  l_i := \max\{m_{ij_1},\dots,m_{ij_{\nu_i}}\}
  \quad \textrm{and} \quad
  \alpha_{Q_i} := (a_{i0},\dots,a_{il_i-1})
  $$
  where $a_{ik} := \sum_{e=1}^{\nu_i}\max\{m_{ij_e}-k,0\}$
  with $0\le k \le l_i-1$. By rearranging the elements
  of $(\alpha_{Q_1},\dots,\alpha_{Q_r})$
  in non-increasing order, we get the
  $(l_1+\cdots+l_r)$-tuple which is denoted by~$\alpha_{\Y}$.

\item[(b)] Suppose $\alpha_\Y = (\alpha_1,\dots,\alpha_l)$
  where $l = \sum_{i=1}^{r} l_i$.
  We define the {\em conjugate} of~$\alpha_{\Y}$
  $$
  \alpha^*_\Y = (\alpha^*_1,\alpha^*_2,\dots)
  \quad \textrm{with} \quad
  \alpha^*_i \!=\! \#\{\alpha_j \in \alpha_\Y
  \mid \alpha_j \ge i\}.
  $$
  Here $\alpha^*_i = 0$ for $i > \alpha_1$.

\item[(c)] Similarly,  for every $R_j \in \pi_2(\X)$,
  we set $\Y_{2,R_j} :=
  m_{i_1j}P_{i_1j}+\cdots+m_{i_{\nu'_i}j}P_{i_{\nu'_i}j}$
  where $P_{i_kj} = Q_{i_k}\times R_{j}$ are those points
  of~$\X$ whose second coordinate is $R_j$. We also set
  $$
  l'_j := \max\{m_{i_1j},\dots,m_{i_{\nu'_i}j}\}
  \quad \textrm{and} \quad
  \beta_{R_j} := (b_{j0},\dots,b_{jl'_j-1})
  $$
  where $b_{jk} := \sum_{e=1}^{\nu'_i}\max\{m_{i_ej}-k,0\}$ with
  $0\le k \le l'_j-1$.
  We let $\beta_{\Y}$ denote the $(l'_1+\cdots+l'_t)$-tuple
  which is obtained by rearranging the elements of
  $(\beta_{R_1},\dots,\beta_{R_t})$ in non-increasing order.

\item[(d)] Suppose that $\beta_\Y = (\beta_1,\dots,\beta_{l'})$
  where $l' = \sum_{j=1}^{t} l'_j$.
  We define the {\em conjugate} of~$\beta_{\Y}$
  $$
  \beta^*_\Y = (\beta^*_1,\beta^*_2,\dots)
  \quad \textrm{with} \quad
  \beta^*_j = \#\{\beta_k \in \beta_\Y \mid \beta_k \ge j\}.
  $$
  Note that here it is $\beta^*_j = 0$ for $j > \beta_1$.
\end{enumerate}
\end{notation}

With the above notation, we also define
$$
\hat{\alpha}_{\Y} := (a_{10}+\nu_1,\dots,a_{r0}+\nu_r)
\hbox{\quad and \quad }
\hat{\beta}_{\Y} := (b_{10}+\nu'_1,\dots,b_{t0}+\nu'_t).
$$

Let us calculate the tuples of Notation~\ref{notationSec4.3}
in a concrete case.

\begin{example}
Let $\Y =2P_{11}+2P_{12}+P_{23}+P_{31}+2P_{32}$
be the fat point scheme of~Example~\ref{examS3.2}.
We want to compute the tuples $\alpha_\Y$, $\hat{\alpha}_\Y$,
$\alpha^*_\Y$, $\beta_\Y$, $\hat{\beta}_\Y$ and
$\beta^*_\Y$.
We have $\X =\supp(\Y)$, $\pi_1(\X) = \{Q_1,Q_2,Q_3\}$
and $\pi_2(\X)=\{R_1,R_2,R_3\}$. Also, we have
$\Y_{1,Q_1} = 2P_{11}+2P_{12}$. Note that the support
of~$\Y_{1,Q_1}$ contains the points of~$\X$ which lie on
the horizontal ruling defined by $L_{Q_1}$ (see the figure).
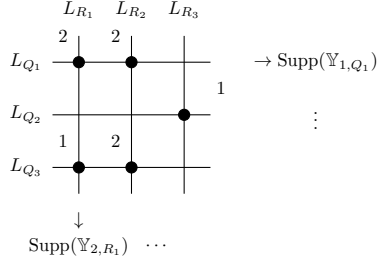
\begin{center}
\begin{tikzpicture}[scale=.7, transform shape]
\filldraw[black] (1,3) circle (3pt);
\node (a11) at (0.7,3.5) {$2$};
\filldraw[black] (2,3) circle (3pt);
\node (a11) at (1.7,3.5) {$2$};
\filldraw[black] (3,2) circle (3pt);
\node (a11) at (3.7,2.5) {$1$};
\filldraw[black] (1,1) circle (3pt);
\node (a11) at (1.7,1.5) {$2$};
\filldraw[black] (2,1) circle (3pt);
\node (a11) at (0.7,1.5) {$1$};
\node (a1) at (0,3) {$L_{Q_1}$};
\node (a2) at (0,2) {$L_{Q_2}$};
\node (a3) at (0,1) {$L_{Q_3}$};
\node (a4) at (1,4) {$L_{R_1}$};
\node (a5) at (2,4) {$L_{R_2}$};
\node (a6) at (3,4) {$L_{R_3}$};
\node (a7) at (5.5,3) {$\rightarrow \supp(\Y_{1,Q_1})$};
\node (b7) at (5.5,2) {$\vdots$};
\node (a81) at (1,0) {$\downarrow$};
\node (a82) at (1,-0.5) {$\supp(\Y_{2,R_1})$};
\node (b8) at (2.5,-0.5) {$\cdots$};
\draw[-] (3,0.5) -- (3, 3.5);
\draw[-] (2,0.5) -- (2, 3.5);
\draw[-] (1,0.5) -- (1, 3.5);
\draw[-] (0.5,3) -- (3.5,3);
\draw[-] (0.5,2) -- (3.5,2);
\draw[-] (0.5,1) -- (3.5,1);
\end{tikzpicture}
\end{center}
We set $l_1 = \max\{m_{11}, m_{12}\} = 2$, $l_2=m_{21}=1$,
and $l_3=\max\{m_{31}, m_{32}\}=2$. Then
$$
\begin{aligned}
a_{10} &= \sum_{e=1}^2 \max\{m_{1e}-0,0\} = 4,\quad
a_{11} &= \sum_{e=1}^2 \max\{m_{1e}-1,0\} = 2, \\
a_{20} &= \sum_{e=1} \max\{m_{2e}-0,0\} = 1,\quad
a_{21} &= \sum_{e=1} \max\{m_{2e}-1,0\} = 0, \\
a_{30} &= \sum_{e=1}^2 \max\{m_{1e}-0,0\} = 3,\quad
a_{31} &= \sum_{e=1}^2 \max\{m_{1e}-1,0\} = 1. \\
\end{aligned}
$$
So, we have $\alpha_{Q_1} = (4,2),\alpha_{Q_2} = (1)$
and $\alpha_{Q_3} = (3,1)$.
Thus we obtain $\alpha_{\Y} = (4,3,2,1,1)$ and
$\hat{\alpha}_\Y=(a_{10}+2, a_{20}+1, a_{30}+2)=(6, 2, 5)$.
Moreover, it follows that $\alpha^*_{\Y}=(5,3,2,1,0,0,\dots)$.

Similarly, for $R_1,R_2,R_3 \in \pi_2(\X)$, we find
$l'_1 = 2$, $l'_2 = 2$ and $l'_3 = 1$, respectively.
Also, we have $\beta_{R_1} = (3,1)$, $\beta_{R_2} = (4,2)$
and $\beta_{R_3} = (1)$.
Hence we get $\beta_{\Y} = (4,3,2,1,1)$ and
$\hat{\beta}_\Y= (5, 6, 2)$.
Also, we have $\beta^*_{\Y} = (5,3,2,1,0,0,\dots)$.
\end{example}

Using these notations, we can give a formula for the
Hilbert function of~$\Omega^1_{R_\Y/K}$,
as the following proposition shows.

\begin{proposition} \label{PropSec3.3}
Let $\Y = \sum_{(i,j)\in D_\X} m_{ij}P_{ij}$
be a fat point scheme in $\popo$ with associated tuples
$\hat{\alpha}_\Y$, $\alpha^*_\Y$, $\hat{\beta}_\Y$, and
$\beta^*_\Y$ as given in Notation~\ref{notationSec4.3}.
\begin{enumerate}
\item[(a)] For all $j \in \N$, if $i \ge l+r-1$ and let
  $h = \min\{j+1,\max\{a \in \hat{\alpha}_\Y\}\}$, then we have
  $$
  \HF_{\Omega^1_{R_\Y/K}}(i,j) =
  4\sum_{k=1}^{j}\alpha^*_k  + 2\alpha^*_{j+1}
  - \sum_{k=1}^{h} \#\{a \in \hat{\alpha}_\Y \mid a \ge k \}.
  $$

\item[(b)] For all $i \in \N$, if $j \ge l'+t-1$ and let
  $h' = \min\{i+1,\max\{b \in \hat{\beta}_\Y\}\}$, then we have
  $$
  \HF_{\Omega^1_{R_\Y/K}}(i,j) =
  4\sum_{k=1}^{i}\beta^*_k  + 2\beta^*_{i+1}
  - \sum_{k=1}^{h'} \#\{b \in \hat{\beta}_\Y \mid b \ge k \}.
  $$
\end{enumerate}
\end{proposition}

\begin{proof}
We need only prove claim~(a) because claim (b) follows
in the same way. Note that $\X = \supp(\Y)$.
For every $Q_i\in \pi_1(\X)$, we write
$\alpha_{Q_i} = (a_{i0},\dots,a_{i\, l_i-1})$ as given
in Notation~\ref{notationSec4.3}.a.
Let $\V= \sum_{(i,j)\in D_\X} (m_{ij}+1)P_{ij}$
be the thickening of~$\Y$ and
$\tilde{\alpha}_\V:=(\alpha'_{Q_1},\dots,\alpha'_{Q_r})$
the tuple associated to $\V$ where
$\alpha'_{Q_i}=(a'_{i0},\dots,a'_{il_i})$. We see that
$$
a'_{i0}
= a_{i0}+\nu_i  \text{ and } a'_{ik}
= \sum_{e=1}^{\nu_i}\max\{m_{ij_e} +1-k,0\}
= a_{i\, k-1}
$$
for $k =1,...,l_i$, and hence
$\alpha'_{Q_i}=(a_{i0}+\nu_i,a_{i0},\dots,a_{i\, l_i-1})$.
Consequently, if
$\tilde{\alpha}_{\Y} := (\alpha_{Q_1},\dots,\alpha_{Q_r})$
is the tuple associated to~$\Y$, and in a non-increasing
ordering we order all the elements of the union
$\{a \mid a\in \tilde{\alpha}_{\Y}\}
\cup \{a_{10}+\nu_1,\dots,a_{r0}+\nu_r\},$
we get the tuple $\alpha_\V$ associated to $\V$. Consider
$\alpha_\Y =(\alpha_1,\dots,\alpha_l), \alpha^*_{\Y}
=(\alpha^*_1,\dots,\alpha^*_{\alpha_1}),$ and
$\alpha_\V = (\alpha'_1,\dots,\alpha'_{l+r})$.

For $i =1,\dots,\alpha_1$, we observe that
$$
\begin{aligned}
(\alpha'_i)^* &= \#\{\alpha'_j \in \alpha_\V
\mid \alpha'_j \ge i\} \\
&= \#\{\alpha_j \in \alpha_\Y \mid \alpha_j \ge i\} +
\#\{a \in \hat{\alpha}_\Y \mid a \ge i\}\\
&= \alpha_i^* +\#\{a \in \hat{\alpha}_\Y \mid a \ge i\}\\
&=\alpha^*_i+\hat{\alpha}^*_i.
\end{aligned}
$$
If $\alpha_1<i\le \alpha'_1$ we have
$$
(\alpha'_i)^*
= \#\{\alpha'_i \in \alpha_\V \mid \alpha'_i \ge i\}
= \#\{\hat{\alpha}_j \in \hat{\alpha}_\Y
\mid {\hat{\alpha}_j} \ge i\}
=\hat{\alpha}^*_i,
$$
and if $i>\alpha'_1$ we have $(\alpha'_i)^*= 0$.
Hence, by \cite[Thm.~3.2]{GV04}, for $j \in\N$
and for $i \ge l+r-1$ we have
$$
\HF_\Y(i-1,j) =  \sum_{k=1}^{j+1}\alpha^*_k
$$
and
$$
\HF_\V(i,j) =  \sum_{k=1}^{j+1}\alpha^*_k +
\sum_{k=1}^{h} \#\{a \in \hat{\alpha}_\Y \mid a \ge k \}
= \sum_{k=1}^{j+1}\alpha^*_k +
\sum_{k=1}^{h}  \hat{\alpha}^*_ k
$$
where $h = \min\{j+1,\max\{a \in \hat{\alpha}_\Y\}\}
=\min\{j+1,{\alpha}'_1\}$.
Therefore, for $j \in\N$ and for $i \ge l+r-1$,
the desired formula of~$\HF_{\Omega^1_{R_{\Y}/K}}$
follows from Corollary~\ref{corSec2.6}.
\end{proof}

Based on this proposition, we can work out certain
values of the Hilbert function of~$\Omega^1_{R_\Y/K}$
explicitly.

\begin{corollary}
Let $\Y = \sum_{(i,j)\in D_\X} m_{ij}P_{ij}$
be a fat point scheme in~$\popo$
and let $\V$ be the thickening of $\Y$.
\begin{enumerate}
\item[(a)] For all $(i,j) \succeq (l+r-1,l'+t-1)$,
  we have
  $$
  \HF_{\Omega^1_{R_\Y/K}}(i,j) =
  4\sum_{(i,j)\in D_\X} \binom{m_{ij}+1}{2}
  - \sum_{(i,j)\in D_\X} (m_{ij}+1).
  $$

\item[(b)] If $i \ge l+r-1$ and $j< l'+t-1$, then
  $$
  \qquad
  \begin{aligned}
  \HF_{\Omega^1_{R_\Y/K}}&(i,j)
  = \HF_{\Omega^1_{R_\Y/K}}(l+r-1,j)\\
  &= 3\HF_\Y(l+r-1,j) + 2\HF_\Y(l+r-1,j-1)-\HF_\V(l+r-1,j).
  \end{aligned}
  $$

\item[(c)] If $i< l+r-1$ and $j \ge l'+t-1$, then
  $$
  \qquad
  \begin{aligned}
  \HF_{\Omega^1_{R_\Y/K}}&(i,j)
  = \HF_{\Omega^1_{R_\Y/K}}(i,l'+t-1) \\
  &= 2\HF_\Y(i-1,l'+t-1) +
  3\HF_\Y(i,l'+t-1) - \HF_\V(i,l'+t-1).
  \end{aligned}
  $$

\item[(d)] For $(i,j) \succeq (l,l')$, we have
  $$
  \qquad \quad
  \HF_{\Omega^1_{R_\Y/K}}(i,j)
  \ge \HF_{\Omega^1_{R_\Y/K}}(i+1,j)
  \, \textrm{and} \,
  \HF_{\Omega^1_{R_\Y/K}}(i,j)
  \ge \HF_{\Omega^1_{R_\Y/K}}(i,j+1).
  $$
\end{enumerate}
\end{corollary}

\begin{proof}
According to Proposition~\ref{PropSec3.3}, we have
$$
\HF_{\Omega^1_{R_\Y/K}}(i,j) =
\begin{cases}
\HF_{\Omega^1_{R_\Y/K}}(l+r-1,l'+t-1) &
\textrm{if}\,\, (i,j) \succeq (l+r-1,l'+t-1),\\
\HF_{\Omega^1_{R_\Y/K}}(l+r-1,j) &
\textrm{if}\,\, i\ge l+r-1, \\
\HF_{\Omega^1_{R_\Y/K}}(i,l'+t-1) &
\textrm{if}\,\, j \ge l'+t-1.
\end{cases}
$$
So, Corollary~\ref{corSec2.6} and \cite[Cor.~3.4]{GV04}
imply claims (b), (c) and
$\HF_{\Omega^1_{R_\Y/K}}(i,j)
= 5\sum_{(i,j)\in D_\X} \binom{m_{ij}+1}{2} -
\sum_{(i,j)\in D_\X} \binom{m_{ij}+2}{2}$
for all $(i,j) \succeq (l+r-1,l'+t-1)$.
Thus claim (a) also follows.
Furthermore, for $(i,j) \succeq (l,l')$ we see that
$\HF_{\Omega^1_{R_\Y/K}}(i,j)
= 5\HF_\Y(i,j) - \HF_\V(i,j)
= 5\sum_{(i,j)\in D_\X} \binom{m_{ij}+1}{2} - \HF_\V(i,j)$
by \cite[Cor.~3.4]{GV04}.
We also have $\HF_\V(i,j) \le \HF_\V(i+1,j)$
and $\HF_\V(i,j) \le \HF_\V(i,j+1)$ for all $(i,j)\in\N^2$
by \cite[Prop.~1.3(i)]{GV04},
and therefore claim~(d) holds true.
\end{proof}

\begin{remark}
The above corollary tells us that if we know the values
of $\HF_{\Omega^1_{R_\Y/K}}(l+r-1,j)$ for $j=0,...,l'+t$
and the values of $\HF_{\Omega^1_{R_\Y/K}}(i,l'+t-1)$
for $i=0,...,l+r$, then we know all but a finite numbers
of values of the Hilbert function of $\Omega^1_{R_\Y/K}$.
\end{remark}

\begin{example}
Let us go back to Example~\ref{examS3.2}.
We see that the fat point scheme~$\Y$ satisfies $l = l' =5$
and  $r=t=3$. For $(i,j) \succeq (7,7)$,
we have
$$
\HF_{\Omega^1_{R_\Y/K}}(i,j) = 31 = 4\sum_{(i,j)\in D_\X}
\binom{m_{ij}+1}{2} - \sum_{(i,j)\in D_\X} (m_{ij}+1).
$$
Moreover, for $(i,j)\succeq (5,5)$, we have
$\HF_{\Omega^1_{R_\Y/K}}(i,j)\ge\HF_{\Omega^1_{R_\Y/K}}(i+1,j)$
and $\HF_{\Omega^1_{R_\Y/K}}(i,j)
\ge \HF_{\Omega^1_{R_\Y/K}}(i,j+1)$.
\end{example}

%
%

\bigbreak
\section{Special Fat Point Schemes in $\popo$}\label{CI}

In this section we study cases in which the support $\X$
of the fat point scheme $\Y$ is either a complete intersection
or an almost complete intersection in~$\popo$.
Let us begin with the complete intersection case.

\begin{definition} Let~$I$ be an ideal
in $S=K[X_0,X_1,Y_0,Y_1]$.
\begin{enumerate}
\item[(a)] The ideal~$I$ is called a
{\it (bihomogeneous) complete intersection}
if it is generated by a (bihomogeneous) regular sequence.

\item[(b)] A set of distinct points~$\X$ in $\popo$ is called
a {\it complete intersection} if its vanishing ideal $I_{\X}$
is a bihomogeneous complete intersection.

\item[(c)] If the vanishing ideal of a set of distinct points
$\X$ in~$\popo$ is generated by two bihomogeneous polynomials
of degrees $(d_1,0)$ and $(0,d_2)$, where $d_1,d_2\ge 1$,
we say that $\X$ is a
complete intersection {\it of type $(d_1,d_2)$}
and write $CI(d_1,d_2)$.
\end{enumerate}
\end{definition}

\begin{remark}
Note that $\X$ is a complete intersection if and only if
it is a $CI(d_1,d_2)$ for some $d_1,d_2\ge 1$
(cf.~\cite[Thm.~1.2]{GMR92}).
Moreover, a complete intersection
is arithmetically Cohen-Macaulay and
\cite[Lemma~2.26]{GV15} provides an explicit description
of the bigraded minimal free resolution of its
bihomogeneous coordinate ring.
\end{remark}

Now we can apply the results of the previous sections to compute
the Hilbert function of $\Omega^1_{R_{\Y}/K}$ in the case that
$\Y$ is an equimultiple fat point scheme whose
support $\X$ is a complete intersection of type~$(d_1,d_2)$.
In the following, unless stated otherwise, we assume that
$\X$ is a set of distinct points in~$\popo$ which is a
$CI(d_1,d_2)$ with $d_1\le d_2$.
For more details about complete intersections in~$\popo$,
see \cite[Sec.~5.4]{GV15}.

The first easy example of a complete intersection is
the case when $\X$ is a $CI(1,d_2)$, i.e.,
when $\X$ is  a set of distinct points on a line.
In the following example we present a formula
for the Hilbert function of~$\Omega^1_{R_{\Y}/K}$
depending only on the tuple $\alpha_\Y$
when $\Y$ is a fat point scheme whose support $\X$ is
a set of $s$ points on a line.

\begin{example}
Let $\Y = m_{11}P_{11}+ m_{12}P_{12}+\cdots + m_{1s}P_{1s}$
be a fat point scheme in $\popo$ whose support is
contained in a line
defined by a form of degree $(1,0)$.
Suppose that $\alpha_\Y = (a_1,\dots,a_l)$ with
$l=\max\{m_{11},\dots,m_{1s}\}$ is the tuple associated
with~$\Y$. Let $\mathcal{Z}_k$ be a matrix of $k$
rows and infinitely many columns with all entries equal
to zero, and let
$$
\mathcal{A}_k =
\begin{bmatrix}
1 & 2 & \cdots & a_k-1 & a_k & a_k & \cdots \\
1 & 2 & \cdots & a_k-1 & a_k & a_k & \cdots \\
1 & 2 & \cdots & a_k-1 & a_k & a_k & \cdots \\
\vdots & \vdots & \vdots & \vdots &\vdots &\vdots &\ddots
\end{bmatrix}
$$
Then we have $\HF_{\Y}=\sum_{k=1}^{l}
\begin{bmatrix}
\mathcal{Z}_{k-1}\\ \mathcal{A}_k
\end{bmatrix}$
by \cite[Thm.~2.2]{GV04}.
Note that the associated tuple of the scheme
$\V = (m_{11}+1)P_{11}+(m_{12}+1)P_{12}+
\cdots+(m_{1s}+1)P_{1s}$
is given by $\alpha_\V = (a_1+s,a_1,\dots,a_l)$.
By letting
$$
\mathcal{A} =
\begin{bmatrix}
1 & 2 & \cdots & a_1+s-1 & a_1+s & a_1+s & \cdots \\
1 & 2 & \cdots & a_1+s-1 & a_1+s & a_1+s & \cdots \\
1 & 2 & \cdots & a_1+s-1 & a_1+s & a_1+s & \cdots \\
\vdots & \vdots & \vdots & \vdots & \vdots & \vdots & \ddots
\end{bmatrix}
$$
we have
$$
\HF_\V = \mathcal{A} + \sum_{k=1}^{l}
\begin{bmatrix}
\mathcal{Z}_{k}\\ \mathcal{A}_k
\end{bmatrix}.
$$
Now we let $\mathcal{B}_k$ be the matrix obtained from
$\mathcal{A}_k$ by adding one zero column to $\mathcal{A}_k$
in the first position. Notice that $\HF_{\Y}(i,j) = 0$
if $(i,j) \nsucceq (0,0)$.
An application of Corollary~\ref{corSec2.6} yields
$$
\begin{aligned}
\HF_{\Omega^1_{R_{\Y}/K}} &= \sum_{k=1}^{l}2
 \begin{bmatrix}
 \mathcal{Z}_{k}\\ \mathcal{A}_k
 \end{bmatrix}
+\sum_{k=1}^{l}
 2\begin{bmatrix}
 \mathcal{Z}_{k-1} \\ \mathcal{B}_k
 \end{bmatrix}
+\sum_{k=1}^{l}
 \begin{bmatrix}
 \mathcal{Z}_{k-1}\\ \mathcal{A}_k
 \end{bmatrix}
- \mathcal{A} - \sum_{k=1}^{l}
 \begin{bmatrix}
 \mathcal{Z}_{k}\\ \mathcal{A}_k
 \end{bmatrix} \\
&= \sum_{k=1}^{l}
 \begin{bmatrix}
 \mathcal{Z}_{k}\\ \mathcal{A}_k
 \end{bmatrix}
+\sum_{k=1}^{l}
 2\begin{bmatrix}
 \mathcal{Z}_{k-1} \\ \mathcal{B}_k
 \end{bmatrix}
+\sum_{k=1}^{l}
 \begin{bmatrix}
 \mathcal{Z}_{k-1}\\ \mathcal{A}_k
 \end{bmatrix}
- \mathcal{A}.
\end{aligned}
$$
\end{example}

It is well-known that a fat
point scheme~$\V$ whose support is on a line is ACM. In this case
we can characterize the Hilbert function of the K\"{a}hler
differential module of~$\Y$ in terms of $\alpha_\Y$ and
$\hat{\alpha}_\Y$ as the preceding example showed.
When the thickening $\V$ of~$\Y$ is an ACM fat point scheme
in~$\popo$, this characterization can be generalized as follows.

\begin{proposition}\label{ACMsupport}
Let $\X$ be a set of distinct points in~$\popo$\!,
let $\Y \!=\!\sum_{(i,j)\in D_\X}\!m_{ij}P_{\!ij}$
be a fat point scheme with associated tuples
$\hat{\alpha}_\Y$ and $\alpha_\Y$, and let $\V$
be the thickening of~$\Y$. Suppose that
$\V$ is ACM and that
$\hat{\alpha}_1 \ge \cdots \ge \hat{\alpha}_r\ge \alpha_1$.
Let $c_{ij} = \min\{j+1, \alpha_{i+1}\}$ for $i=0,...,l-1$
and all $j\in \mathbb{N}$. Then the Hilbert function
of $\Omega^1_{R_{\Y}/K}$ is
$$
\begin{aligned}
\HF_{\Omega^1_{R_{\Y}/K}}(i,j) &=
2\sum_{k=0}^{\min\{i-1,l-1\}}c_{kj} + 2
\sum_{k=0}^{\min\{i,l-1\}}c_{k\, j-1}
+ \sum_{k=0}^{\min\{i,l-1\}} c_{kj} \\
&\quad \ - \sum_{k=0}^{\min\{i,r-1\}} \!\!\min\{j+1,
\hat{\alpha}_{k+1}\}-\sum_{k=r}^{\min\{i,l+r-1\}}c_{k-r\,j}
\end{aligned}
$$
for all $(i,j) \in \mathbb{N}^2$.
\end{proposition}

\begin{proof}
Clearly, we have $\alpha_\V =
(\hat{\alpha}_1,\dots,\hat{\alpha}_r,\alpha_1,\dots,\alpha_l)$.
So, by \cite[Cor.~4.10]{GV04}, we get
$$
\HF_{\V}(i,j) = \sum_{k=0}^{\min\{i,r-1\}}\min\{j+1,
\hat{\alpha}_{k+1}\}+\sum_{k=r}^{\min\{i,l+r-1\}}c_{k-r\,j}.
$$
Note that $\V$ and $\Y$ have the same support and that~$\V$ is ACM.
Hence~$\Y$ is ACM by \cite[Thm.~4.2]{GV04} and by
the fact that $\mathcal{S}_\Y \subseteq \mathcal{S}_\V$,
where $\mathcal{S}_\Y$ is the set of $t$-tuples
$\mathcal{S}_\Y = \{(\max\{0,m_{i1}-k\},\dots,
\max\{0,m_{it}-k\}) \mid 1\le i\le r,
k\in \mathbb{N}, m_{ij}=0 \
\textrm{if}\ (i,j)\notin D_\X\}$
and $\mathcal{S}_\V$ is defined in the same way.
Again \cite[Cor.~4.10]{GV04} yields that
$$
\HF_{\Y}(i,j) = \sum_{k=0}^{\min\{i,l-1\}} c_{kj}.
$$
Hence the claim follows from Corollary~\ref{corSec2.6}.
\end{proof}

\begin{remark}
Let us observe that if the thickening $\V$ of~$\Y$
is ACM then it follows that also $\Y$ is ACM.
The converse is not true in general, since if $\Y$
is an ACM set of distinct points whose support is not
a complete intersection then its thickening $\V$ is
not ACM.

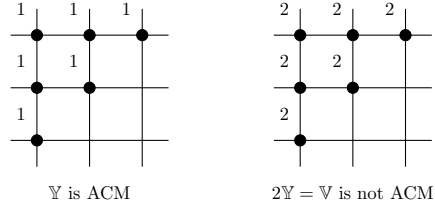
\begin{center}
\begin{tikzpicture}[scale=.7, transform shape]
\filldraw[black] (1,3) circle (3pt);
\node (a11) at (0.7,3.5) {$1$};
\filldraw[black] (2,3) circle (3pt);
\node (a11) at (1.7,3.5) {$1$};
\filldraw[black] (3,3) circle (3pt);
\node (a11) at (2.7,3.5) {$1$};
\filldraw[black] (1,2) circle (3pt);
\node (a11) at (0.7,2.5) {$1$};
\filldraw[black] (2,2) circle(3pt);
\node (a11) at (1.7,2.5) {$1$};
\filldraw[black] (1,1) circle (3pt);
\node (a11) at (0.7,1.5) {$1$};
\draw[-] (3,0.5) -- (3,3.5);
\draw[-] (2,0.5) -- (2, 3.5);
\draw[-] (1,0.5) -- (1, 3.5);
\draw[-] (0.5,3) -- (3.5,3);
\draw[-] (0.5,2) -- (3.5,2);
\draw[-] (0.5,1)-- (3.5,1);
\node (a1) at (2,0) {$\Y$ is ACM};

\filldraw[black] (6,3) circle (3pt);
\node (a11) at (5.7,3.5) {$2$};
\filldraw[black] (7,3) circle (3pt);
\node (a11) at (6.7,3.5) {$2$};
\filldraw[black] (8,3) circle (3pt);
\node (a11) at (7.7,3.5) {$2$};
\filldraw[black] (6,2) circle (3pt);
\node (a11) at (5.7,2.5) {$2$};
\filldraw[black] (7,2) circle(3pt);
\node (a11) at (6.7,2.5) {$2$};
\filldraw[black] (6,1) circle (3pt);
\node (a11) at (5.7,1.5) {$2$};
\draw[-] (6,0.5) -- (6, 3.5);
\draw[-] (7,0.5) -- (7, 3.5);
\draw[-] (8,0.5) -- (8, 3.5);
\draw[-] (5.5,3) -- (8.5,3);
\draw[-] (5.5,2) -- (8.5,2);
\draw[-] (5.5,1)-- (8.5,1);
\node (a1) at (7,0) {$2\Y=\V$ is not ACM};
\end{tikzpicture}
\end{center}
\end{remark}

\begin{proposition}\label{equimultipleoverCI}
Let $\X$ be a reduced $CI(d_1,d_2)$ with $d_1\le d_2$.
\begin{enumerate}
\item[(a)] For $m\geq 1$, we have the following
exact sequence of graded $R_{m\X}$-modules
$$
0\longrightarrow I^m_{\X}/I^{m+1}_{\X}
\longrightarrow
(S/I^m_{\X})^2(-1,0)\oplus (S/I^m_{\X})^2(0,-1)
\longrightarrow \Omega^1_{R_{m\X}/K}
\longrightarrow 0.
$$

\item[(b)] The Hilbert function of $\Omega^1_{R_{m\X}/K}$ is
$$
\begin{aligned}
\HF_{\Omega^1_{R_{m\X}/K}}(i,j) &=
2\sum_{k=0}^{\min\{i-1,md_1-1\}}c_{kj} +
2 \sum_{k=0}^{\min\{i,md_1-1\}}c_{k\, j-1}
+ \sum_{k=0}^{\min\{i,md_1-1\}} c_{kj} \\
&\quad \ - \sum_{k=0}^{\min\{i,d_1-1\}}\!\!\min\{j+1,(m+1)d_2\}
- \sum_{k=d_1}^{\min\{i,(m+1)d_1-1\}}c_{k-d_1\, j}
\end{aligned}
$$
for all $(i,j) \in \mathbb{N}^2$.
\end{enumerate}
\end{proposition}

\begin{proof}
Since $\X$ is a complete intersection in $\popo$,
using \cite{ZS60}, App.~6, Lemma 5, we have that
$I_{\X}^{(m)}=I_{\X}^{m}$ for all $m\ge 1$.
Hence claim~(a) follows from Theorem~\ref{generpropSec2.5}.

To prove~(b), suppose $\X$ is a $CI(d_1,d_2)$ with $d_1\le d_2$
and $\Y=m\X$ with $m\geq 1$. Then $\Y$ is ACM
by~\cite[Thm.~4.2]{GV04}.
By our hypotheses, the associated tuples of~$\Y$ are
$$
\hat{\alpha}_\Y=(\underbrace{(m+1)d_2,\dots,(m+1)d_2}_{d_1})
\quad \mbox{and}\quad
\alpha_\Y =
(\underbrace{md_2,\dots,md_2}_{d_1},\dots,
\underbrace{d_2,\dots,d_2}_{d_1}).
$$
Here we have $l =md_1$, $\V=(m+1)\X$, and $(m+1)d_2\geq d_2$.
Let $c_{ij} = \min\{j+1, \alpha_{i+1}\}$ for $i=0,...,md_1-1$
and all $j\in \mathbb{N}$. An application of
Proposition~\ref{ACMsupport} yields the conclusion.
\end{proof}

Recall that the {\it first difference function} of
a numerial function $H:\Z^2\rightarrow \Z$, denoted
by $\Delta H$, is the function
$\Delta H:\mathbb{Z}^2 \rightarrow \mathbb{Z}$
defined by
\[
\Delta H(i,j) := H(i,j)-H(i-1,j)-H(i,j-1)+H(i-1,j-1)
\]
where $H(i,j) = 0$ if $(i,j) \not\succeq (0,0)$.

With the above results and~\cite[Thm.~4.2]{GV04},
we can describe the shape
of the first difference function of the Hilbert function
of $\Omega^1_{R_{m\X}/K}$ in the ACM case in terms
of $d_1,d_2$ and $m$.

\begin{corollary}\label{Delta}
Let $\X$ be a $CI(1,d_2)$ consisting of $d_2$ distinct points
on a line, and let $m\ge 1$.
Then the first difference function
of $\HF_{\Omega^1_{R_{m\X}}}$ can be written as
\begin{small}
$$
\left[
\begin{array}{rrrrrrrrrrrrrrrrrrr}
 \underbrace{
   \begin{array}{rrrrrrrrrr}    
    0 & 2 & 2& 2 & 2 &\dots   &\dots & 2 \\
   \end{array}
   }_{md_2}
   \underbrace{
      \begin{array}{rrrrrr}    
       1 & -1 & \dots &  -1 \\
      \end{array}
   }_{d_2}
       \begin{array}{rrrrrrrrr}    
        0 &\dots & \dots& 0 \\
       \end{array}
   \\
 \underbrace{
   \begin{array}{rrrrrrrrrr}    
    2 &4&4& \dots   &\dots & 4 \\
   \end{array}
   }_{(m-1)d_2}
   \underbrace{
     \begin{array}{rrrrrr}    
      3 & 1 & \dots &  1 \\
     \end{array}
   }_{d_2}
      \begin{array}{rrrrrrrrrrrrr}    
       0 &\dots&\dots &\dots&\dots& 0 \\
      \end{array}
    \\
   \begin{array}{rrrrrrrrrrrrrrrr}
     \ddots &      \ddots &      \ddots &
     \ddots &       \ddots &     \ddots &
     & \ddots &    & \ddots &    \ddots &
     \ddots &     \\
   \end{array}
    \\
  \underbrace{
    \begin{array}{rrrrrrrrrr}    
     2 & 4&\dots&\dots&\dots & 4 \\
    \end{array}
    }_{d_2}
    \underbrace{
     \begin{array}{rrrrrr}    
       3 & 1 & \dots &  1 \\
     \end{array}
     }_{d_2}
      \begin{array}{rrrrrrrrrrrrr}    
       0 &\dots &\dots&\dots& && 0 \\
      \end{array}
    \\
  \underbrace{
    \begin{array}{rrrrrrrrrr}    
     1 & 1&\dots & 1 \\
    \end{array}
    }_{d_2}
    \begin{array}{rrrrrrrrrrrrrrrrrrrrrr}  
     0 & 0 & 0& 0&\dots& &\dots& & & & & & & & &\dots & 0\\
    \end{array}
    \\
     \begin{array}{rrrrrrrrrrrrrrrrrrrrrrrrrr} 
      0 & 0 &\dots&\dots&&&& &&&& &&&& &&&& &&&&\dots& 0\\
     \end{array}
\end{array}\right].
$$
\end{small}
\end{corollary}

\begin{corollary}\label{Delta2}
Let $\X$ be a $CI(d_1,d_2)$ with $1<d_1$, and let $m\ge 1$.
Then the first difference function of
$\HF_{\Omega^1_{R_{m\X}}}$ can be written as
\begin{small}
$$
\left[
\begin{array}{rrrrrrrrrrrrrrrrrrr}
 \underbrace{
       d_1 \left\{ \begin{array}{rrrrrrrrrr}    
      0 & 2 & 2& 2 & 2 &\dots   &\dots & 2 \\
      2 & 4 & 4 & 4 & 4 & \dots &\dots& 4 \\
          & \ddots &   \\
       2 & 4 & 4 & 4 & 4 & \dots  &\dots& 4  \\
      \end{array}\right.
  }_{md_2}
   \underbrace{
         \begin{array}{rrrrrr}    
       1 & -1 & \dots &  -1 \\
     1 &-1 & \dots & -1 \\
          & \ddots &   \\
       1 &-1 & \dots & -1  \\
      \end{array}
  }_{d_2}
       \begin{array}{rrrrrrrrr}    
       0 &\dots & 0 \\
     0& \dots &0 \\
          & \ddots &   \\
       0&\dots & 0  \\
      \end{array}
    \\
     \underbrace{
     d_1\, \left\{\begin{array}{rrrrrrrrrr}    
      2 & 4& \dots   &\dots & 4 \\
      2 & 4 &  \dots &\dots& 4 \\
          & \ddots &   \\
       2 & 4   &\dots&\dots & 4  \\
      \end{array}\right.
      }_{(m-1)d_2}
    \underbrace{
         \begin{array}{rrrrrr}    
       3 & 1 & \dots &  1 \\
     1 &-1 & \dots & -1 \\
          & \ddots &   \\
       1 &-1 & \dots & -1  \\
      \end{array}
  }_{d_2}
   \begin{array}{rrrrrrrrrrrrr}    
       0 &\dots &\dots&\dots& 0 \\
     0& \dots &\dots&\dots&0 \\
          & \ddots &   \\
       0&\dots &\dots&\dots& 0  \\
      \end{array}
    \\
     \begin{array}{rrrrrrrrrrrrrrrr}
       \ddots &\ddots &\ddots &\ddots &\ddots &\ddots & &
       \ddots &    & \ddots & \ddots &  \ddots &     \\
       \end{array}\\

       \underbrace{
    d_1\ \left\{
    \begin{array}{rrrrrrrrrr} 
      2 & 4&\dots & 4 \\
      2 & 4 &\dots& 4 \\
          & \ddots &   \\
       2 & 4 &\dots & 4  \\
    \end{array}\right.
    }_{d_2}
    \underbrace{
      \begin{array}{rrrrrr}    
       3 & 1 & \dots &  1 \\
       1 &-1 & \dots & -1 \\
          & \ddots &   \\
       1 &-1 & \dots & -1  \\
      \end{array}
    }_{d_2}
      \begin{array}{rrrrrrrrrrrrr}    
       0 &\dots &\dots&\dots& && 0 \\
       0 &\dots &\dots&\dots& && 0 \\
          & \ddots &   \\
       0 &\dots &\dots&\dots& && 0 \\
      \end{array}
    \\
    \underbrace{  d_1\,\left\{
      \begin{array}{rrrrrrrrrr}    
      1 & 1&\dots & 1 \\
      -1 & -1 &\dots& -1 \\
          & \ddots &   \\
      -1 & -1 &\dots & -1  \\
      \end{array}\right.
      }_{d_2}
      \begin{array}{rrrrrrrrrrrrrrrrrrrrrr} 
       0 & 0 & \dots & & & & & & & & &  & &\dots  & 0 \\
       0 & 0 & \dots & & & & & & & & & & &\dots  & 0 \\
          & \ddots &   \\
       0 & 0 & \dots & & & & & & & & & & &\dots  & 0 \\
      \end{array}
    \\
      \begin{array}{rrrrrrrrrrrrrrrrrrrrrrrrrr} 
       0 & 0 & \dots & & & &&&&& &&&&&&& & & & & & &\dots& 0 \\
       0 & 0 & \dots & & & &&&&& &&&&&&& & & & & & &\dots& 0 \\
          & \ddots &   \\
       0 & 0 & \dots & & & &&&&& &&&&&&& & & & & & &\dots& 0 \\
      \end{array}
\end{array}\right].
$$
\end{small}
\end{corollary}

\begin{example}\label{sepDiff}
Let $i\ge 0$, let $Q_i = R_i = [1:i] \in \P^1$, and
let $P_{ij}$ denote the point $Q_i\times R_j$ in~$\popo$.
We let $\X$ be the set of points
$\X =\{P_{11},P_{12},P_{13},P_{21},P_{22},P_{23}\}$
in~$\popo$. Clearly, the set~$\X$ is a $CI(2,3)$.
For $\Y = 3\X$, the first difference function
of~$\HF_{\Omega^1_{R_{\Y}/K}}$ is
\begin{small}
$$
\left[\begin{array}{cccccccccccccc}
0 & 2 & 2 & 2 & 2 & 2 & 2 & 2 & 2 & 1 & -1 & -1 &0 & \dots \\
2 & 4 & 4 & 4 & 4 & 4 & 4 & 4 & 4 & 1 & -1 & -1 & 0 &\dots \\
2& 4 & 4 & 4 & 4 & 4 & 3 & 1 & 1 & 0 & 0 & 0 & 0 &\dots \\
2 & 4 & 4 & 4 & 4 & 4 &1 &-1 & -1 & 0 & 0 & 0 & 0 &\dots \\
2 & 4 & 4 & 3 & 1& 1 & 0 & 0 & 0 & 0 & 0 & 0& 0 &\dots \\
2 & 4 & 4 & 1 & -1 & -1 & 0 & 0 & 0 & 0 & 0 & 0& 0 &\dots \\
1 & 1 &1 & 0 & 0 & 0 & 0 & 0 & 0 & 0 & 0 & 0
& 0 &\dots \\
-1 & -1 & -1 & 0 & 0 & 0 & 0 & 0 & 0 & 0 & 0 & 0& 0 &\dots \\
0 & 0 & 0 & 0 & 0 & 0 & 0 & 0 & 0 & 0 & 0 & 0& 0 &\dots \\
\vdots & \vdots & \vdots &\vdots & \vdots & \vdots &
\vdots &\vdots & \vdots & \vdots & \vdots & \vdots &
\vdots & \ddots
\end{array}\right].
$$
\end{small}
\end{example}

For equimultiple fat point schemes supported at a complete
intersection, the Hilbert function of the K\"ahler differential
module exhibits the following type of uniform behaviour.

\begin{proposition} \label{indip1}
Let $\X$ be a $CI(d_1,d_2)$ consisting
of $s=d_1d_2$ distinct points in~$\popo$, and let $m\geq 2$.
For $1\le i\le d_1$ and $1\le j \le d_2$, let
$\Y_{ij}$ be the fat point scheme $\Y_{ij} =
\sum_{(k,l)\in D_\X\setminus\{(i,j)\}} mP_{kl}+(m-1)P_{ij}$
supported at~$\X$.
Then the Hilbert function of $\Omega^1_{R_{\Y_{ij}}/K}$
does not depend on the choice of $(i,j)$.
\end{proposition}

\begin{proof}
Since the equimultiple fat point scheme $m\X$ is ACM,
the fat point schemes $\Y_{ij}$ are also ACM.
In this case we have
$D_\X = \{(i,j)\in \N^2\mid 1\le i\le d_1, 1\le j\le d_2\}$.
So, Theorem~\ref{degP1xP1} implies that
$$
\deg_{m\X}(P_{ij}) =  ((d_1-1,md_2-1),
(2d_1-1,(m-1)d_2-1),\dots, (md_1-1, d_2-1))
$$
for all $P_{ij}\in\X$.
According to \cite[Appl.~3.6]{GV12},
we have
$$
\HF_{\Y_{ij}}(k,l) = \HF_{m\X}(k,l)-
\#\{(k',l') \in \deg_{m\X}(P_{ij})
\mid (k',l') \preceq (k,l)\}
$$
for all $(k,l)\in \N^2$, and thus the fat point
schemes $\Y_{ij}$ all have the same Hilbert function.
Similarly, also the fat point schemes
$\V_{ij}=\sum_{(k,l)\in D_\X\setminus\{(i,j)\}}
(m+1)P_{kl}+mP_{ij}$
all have the same Hilbert function.
By Theorem~\ref{generpropSec2.5}, the sequence of
bigraded $R_{\Y_{ij}}$-modules
$$
0\longrightarrow
I_{\Y_{ij}}/I_{\V_{ij}} \longrightarrow
R_{\Y_{ij}}^2(-1,0)\oplus R_{\Y_{ij}}^2(0,-1)
\longrightarrow \Omega^1_{R_{\Y_{ij}}/K}
\longrightarrow 0
$$
is exact. Hence we have
$$
\HF_{\Omega^1_{R_{\Y_{ij}}/K}}(k,l)
=2\HF_{\Y_{ij}}(k-1,l)+2\HF_{\Y_{ij}}(k,l-1)
+\HF_{\Y_{ij}}(k,l)-\HF_{\V_{ij}}(k,l)
$$
for all $(k,l)$. The conclusion follows.
\end{proof}

If we remove the hypothesis $m\ge 2$ in the preceding
proposition, we are still able to show the following
result for complete intersections~$\X$ of type $(h,h)$.

\begin{proposition} \label{indip2}
Let $h\ge 1$, let $\X$ be a $CI(h,h)$ in~$\popo$,
and for $(i,j)\in D_{\X}$ let
$\X_{ij}=\X\setminus \{P_{ij}\}$. Then the
Hilbert function of $\Omega^1_{R_{\X_{ij}}/K}$ does not
depend on the choice of~$(i,j)$.
\end{proposition}

\begin{proof}
For $i,j=1,\dots, h$, we let $\W_{ij}$ and
$\Y_{ij}$  be fat point schemes
$$
\W_{ij}=
{\textstyle \sum\limits_{(k,l)\in D_\X\setminus\{(i,j)\}} }
2P_{kl} + P_{ij}
\quad \mbox{and}\quad
\Y_{ij}=
{\textstyle \sum\limits_{(k,l)\in D_\X\setminus\{(i,j)\}} }
2P_{kl}.
$$
Note that $\supp(\W_{ij}) = \X$ and
$\supp(\Y_{ij}) = \X_{ij}$.
We can argue as in the proof of Proposition~\ref{indip1}
to get that the Hilbert function
of~$\W_{ij}$ does not depend on $(i,j)$.
Moreover, since $\X$ is complete intersection,
Theorem~\ref{pnmsepfrombetti} shows that
all the subsets $\X_{ij}$ have the same Hilbert function.
Let $f^*\in R_{\W_{ij}}$ be the minimal separator
of $P_{ij}$ in $\W_{ij}$.
By Theorem~\ref{pnmsepfrombetti},
it is unique since the
multiplicity of~$P_{ij}$ in~$\W_{ij}$ is $1$.
Let $F^* \in S$ be a representative of $f^*$.
Using Theorem \ref{degP1xP1}, we have that
$\deg(F^*)=(2(h-1),2(h-1))$ and this bidegree does not depend
on the choice of $(i,j)$. Since the support $\X$ is a complete
intersection, we can always renumber the lines such that
we have $(i,j)=(h,h)$, i.e., $P_{ij}=P_{h h}$.
Thus the Hilbert functions of the schemes $\Y_{ij}$ do
not depend of the choice of~$(i,j)$. Therefore an application
of~Corollary~\ref{corSec2.6} yields the desired property.
\end{proof}

\begin{example}
Consider Example \ref{sepDiff}, and reduce the
multiplicity of $P_{11}$ by one, i.e., let~$\Y_{11}$ be the subscheme
$\Y_{11} = 2P_{11}+3P_{12}+3P_{13}+3P_{21}+3P_{22}+3P_{23}$
of~$\Y$. Then the first difference
function $\Delta \HF_{\Omega^1_{R_{\Y_{11}}/K}}$ is given by
\begin{tiny}
$$
\left[\begin{array}{cccccccccccccc}
0 & 2 &2 & 2 & 2 & 2 & 2 & 2 & 2 & 1 & -1 & -1 & 0& \dots \\
2 & 4 & 4 & 4 & 4 & 4 & 4 & 4 & 3 & -1 &-1 & 0 & 0& \dots \\
2 & 4 & 4 & 4 & 4 & 4 & 3 &1 & -1 & 0 & 0 & 0 & 0& \dots \\
2 & 4 & 4 & 4 & 4 & 3 & -1 & -1 & 0 & 0 & 0 & 0 & 0& \dots \\
2 & 4 & 4 & 3 & 1 & -1 & 0 & 0 & 0 & 0 & 0 & 0& 0& \dots \\
2 & 4 & 3 & -1 & -1 & 0 & 0 & 0 & 0 & 0 & 0 & 0& 0& \dots \\
1 & 1 & -1 & 0 &0 & 0 & 0 & 0 & 0 & 0 & 0 & 0& 0 & \dots \\
-1 & -1 & 0 & 0 & 0 & 0 & 0 & 0 & 0 & 0 & 0 & 0& 0 & \dots \\
0 &0 & 0 & 0 & 0 & 0 & 0 & 0 & 0 & 0 & 0 & 0& 0 & \dots \\
\vdots & \vdots & \vdots &\vdots & \vdots & \vdots &\vdots &
\vdots &\vdots & \vdots & \vdots & \vdots & \vdots & \ddots
\end{array}\right].
$$
\end{tiny}
Moreover, every subscheme $\Y_{ij}$ of $\Y$, obtained by reducing
in $3\X$ the multiplicity of one point $P_{ij}$ by one, satisfies
that $\HF_{\Omega^1_{R_{\Y_{ij}}/K}} =
\HF_{\Omega^1_{R_{\Y_{11}}/K}}$.
\end{example}

In the final part of this section we look at the
Hilbert function of the K\"{a}hler
differential module of an equimultiple fat point scheme
whose support is an almost complete intersection.

\begin{definition}
A set of points $\X$ in $\popo$ is called an
{\it almost complete intersection (ACI)} if the number of minimal
generators of $I_\X$ is one more than the codimension,
i.e., if $I_{\X}$ has exactly three minimal generators.
\end{definition}

For a set of points in~$\popo$, the property of being an ACI
can be characterized as follows.

\begin{proposition}\label{charACI}
Let $\X$ be a set of points in~$\popo$
Then the following conditions are equivalent.
\begin{enumerate}
\item[(a)] The set~$\X$ is an almost complete intersection.

\item[(b)] The set~$\X$ is ACM and we have
$\alpha_{\X}=(\underbrace{d_1,\dots,d_1}_{a},
\underbrace{d_2,\dots,d_2}_{b})$
with $d_1>d_2$.
\end{enumerate}
Moreover, if~$\X$ is not a complete intersection,
these conditions are equivalent to:
\begin{enumerate}
\item[(c)] The set~$\X$ is ACM and
we have $I_{\X}^{(m)}=I_{\X}^{m}$ for all $m\ge 1$.

\item[(d)] The set~$\X$ is ACM and we have $I_{\X}^{(3)}=I_{\X}^{3}$.
\end{enumerate}
\end{proposition}

\begin{proof}
To show that~(a) implies~(b), we first note that an ACI
set of points~$\X$ is ACM by~\cite[Lemma 2.5]{Pet98}, 
since~$\X$ is obviously a local complete intersection.
The claimed shape of~$\alpha_{\X}$ follows then from
\cite[Cor. 5.6]{GV15}.
The converse implication follows also from
\cite[Cor. 5.6]{GV15}, since that corollary implies that~$I_{\X}$
is minimally generated by three homogeneous polynomials.

Next we prove that~(b) implies~(c).
By~\cite[Sec.~1]{FG17}, the set~$\X$ is of the type
(1.1) defined there, where $m_{11}=m_{21}=m_{12}=1$.
Hence Thm.~4.2 of this paper yields the claim.

Since~(c) is equivalent to~(d) by~\cite[Thm.~8.8]{GV15},
it remains to prove that~(d) implies~(a). This follows
from~\cite[Cor.~1.2]{Co16}.
\end{proof}

Now we are ready to determine the Hilbert function
of the module of K\"ahler differentials of
an equimultiple fat point scheme in~$\popo$ supported at an ACI 
set of points explicitly.

\begin{proposition}
Let $\X$ be an almost complete intersection of
distinct points in $\popo$ whose associated tuple is
$\alpha_{\X}=(\underbrace{d_1,\dots,d_1}_{a},
\underbrace{d_2,\dots,d_2}_{b})$
with $d_1>d_2$.
\begin{enumerate}
\item[(a)] For every $m\ge 1$, we have the following exact
sequence of bigraded $R_{m\X}$-modules
$$
0\longrightarrow I^m_{\X}/I^{m+1}_{\X} \longrightarrow
(S/I^m_{\X})^2(-1,0)\oplus (S/I^m_{\X})^2(0,-1)
\longrightarrow \Omega^1_{R_{m\X}/K}\longrightarrow 0.
$$

\item[(b)] For every $j\in\mathbb{N}$ and every $i \ge (m+1)(a+b)-1$,
we have
$$
\HF_{\Omega^1_{R_{m\X}/K}}(i,j) =
4\sum_{k=1}^{j}\alpha^*_k  + 2\alpha^*_{j+1}
- \delta_j
$$
where $\alpha^*_{m\X} = (\alpha^*_1,\alpha^*_2,\dots)$ with
$\alpha^*_i = \#\{\alpha_j \in \alpha_{m\X}
\mid \alpha_j \ge i\}$, and where
$$
\delta_j :=
\begin{cases}
(j+1)(a+b) & \mbox{\ if\ $j<(m+1)d_2$},\\
(j+1)a+(m+1)bd_2 & \mbox{\ if\ $(m+1)d_2\le j< (m+1)d_1$},\\
(m+1)(ad_1+bd_2) & \mbox{\ if\ $j\ge (m+1)d_1$}.\\
\end{cases}
$$
\end{enumerate}
\end{proposition}

\begin{proof}
By Proposition~\ref{charACI}, we have $I_{\X}^{(m)}=I_{\X}^{m}$
for all $m\ge 1$.
So, claim (a) follows from Theorem~\ref{generpropSec2.5}.
Moreover, in this case the equimultiple fat point
scheme $\Y=m\X$ is not ACM and its associated tuple
is of the form
$$
\alpha_{\Y} =
(\underbrace{md_1,\dots,md_1}_{a},
\dots,\underbrace{d_2,\dots,d_2}_{b}).
$$
In particular, we have $r=a+b$, $l=m(a+b)$, $t=d_1$
and $l'=md_1$. 
Hence we may apply Proposition~\ref{PropSec3.3} with
$\delta_j = \sum_{k=1}^{h} \#\{a \in \hat{\alpha}_{\Y} \mid a \ge k \}$,
where $h = \min\{j+1,(m+1)d_1\}$, and get claim~(b).
\end{proof}

This proposition allows us to calculate 
many Hilbert functions of K\"ahler differential
modules which would otherwise be beyond the reach
of computer algebra systems. Let us see a comparatively
easy example.

\begin{example}
Let $i\ge 0$, let $Q_i = R_i = [1:i] \in \P^1$, and let
$P_{ij}$ denote the point $Q_i\times R_j$ in~$\popo$.
We let $\X$ be the almost complete intersection
$\X =\{P_{11},P_{12},P_{13},P_{21},P_{22},
P_{23}, P_{31},P_{32}\}$
in~$\popo$, and let $\Y= 3\X$ be the equimultiple
fat point scheme in $\popo$ supported at~$\X$.
Then we have $\alpha_{\X} = (3,3,2)$ and
$\alpha_{\Y} = (9,9,6,6,6,4,3,3,2)$.
Thus the Hilbert function $\HF_{\Omega^1_{R_{\Y}/K}}$
is given by
\begin{tiny}
$$
\left[
\begin{matrix} 
0 & 2 & 4 & 6 & 8 & 10 & 12 & 14 & 16 & 17 & 16 & 
15 & 15 & 15 &\cdots \\
2 & 8 & 14 & 20 & 26 & 32 & 38 & 44 & 50 & 52 & 50 & 
48 & 48 & 48 &\cdots \\
4 & 14 & 24 & 34 & 44 & 54 & 64 & 74 & 83 & 84 & 81 & 
79 & 79 & 79 &\cdots \\
6 & 20 & 34 & 48 & 62 & 76 & 89 & 100 & 108 & 109 & 106 & 
104 & 104 & 104 &\cdots \\
8 & 26 & 44 & 62 & 80 & 98 & 112 & 121 & 126 & 127 & 125 & 
123 & 123 & 123 &\cdots \\
10 & 32 & 54 & 76 & 98 & 119 & 132 & 138 & 143 & 144 & 142 & 
140 & 140 & 140 &\cdots \\
12 & 38 & 64 & 89 & 112 & 132 & 144 & 148 & 153 & 155 & 153 & 
151 & 151 & 151 &\cdots \\
14 & 44 & 74 & 100 & 121 & 138 & 148 & 153 & 158 & 160 & 158 & 
156 & 156 & 156 &\cdots \\
16 & 50 & 83 & 108 & 126 & 143 & 153 & 158 & 164 & 166 & 164 & 
162 & 162 & 162 &\cdots \\
17 & 52 & 84 & 109 & 127 & 144 & 155 & 160 & 166 & 168 & 166 & 
164 & 164 & 164 &\cdots \\
16 & 50 & 81 & 106 & 125 & 142 & 153 & 158 & 164 & 166 & 164 & 
162 & 162 & 162 &\cdots \\
15 & 48 & 79 & 104 & 123 & 140 & 151 & 156 & 162 & 164 & 162 & 
160 & 160 & 160 &\cdots \\
15 & 48 & 79 & 104 & 123 & 140 & 151 & 156 & 162 & 164 & 162 & 
160 & 160 & 160 &\cdots \\
15 & 48 & 79 & 104 & 123 & 140 & 151 & 156 & 162 & 164 & 162 & 
160 & 160 & 160 &\cdots\\
\svdots & \svdots & \svdots & \svdots & \svdots &
\svdots & \svdots & \svdots & \svdots & \svdots &
\svdots & \svdots & \svdots & \svdots & \sddots
\end{matrix}
\right].
$$
\end{tiny}
\end{example}

%
%

\bigbreak
\section{The K\"ahler Different for
a Fat Point Scheme in~$\popo$}

In this section we examine the initial Fitting ideal
of the K\"{a}hler differential module, commonly called the
K\"ahler different, for a fat point scheme in~$\popo$.
In the following we restrict our attention to ACM fat
point schemes~$\Y$ in~$\popo$.
In this case we may assume that $x_0,y_0$ is a regular
sequence in~$R_\Y$. We let $R_o = K[x_0,y_0]$.
Then the algebra $R/R_o$ is finite and the monomorphism
$R_o\hookrightarrow R$ defines a Noether normalization.

Suppose that $\{F_1,\dots,F_r\}$ is a bihomogeneous set
of generators of~$I_{\Y}$.
According to \cite[Prop.~4.19]{Ku86},
the K\"{a}hler differential module $\Omega^1_{R_\Y/R_o}$
has a presentation
\begin{equation*}\label{formulaKae.01}
0\longrightarrow \mathcal{K} \longrightarrow
R_\Y dX_1\oplus R_\Y dY_1
\longrightarrow \Omega^1_{R_\Y/R_o}\longrightarrow 0
\end{equation*}
where $\mathcal{K}$ is generated by the elements
$\frac{\partial F_{j}}{\partial x_1}dX_1
+\frac{\partial F_{j}}{\partial y_1}dY_1$
such that $j\in\{1,\dots,r\}$.
The Jacobian matrix
$$
\begin{bmatrix}
\frac{\partial F_{1}}{\partial x_1}& \cdots &
\frac{\partial F_{r}}{\partial x_1} \\
\frac{\partial F_{1}}{\partial y_1}& \cdots &
\frac{\partial F_{r}}{\partial y_1}
\end{bmatrix}
$$
is a relation matrix of $\Omega^1_{R_\Y/R_o}$
with respect to $\{dx_1,dy_1\}$. The initial Fitting
ideal $F_0(\Omega^1_{R_\Y/R_o})$ of~$\Omega^1_{R_\Y/R_o}$
is the bihomogeneous ideal of~$R_\Y$ generated by all
2-minors of the Jacobian matrix.

\begin{definition}
The initial Fitting ideal of~$\Omega^1_{R_\Y/R_o}$
is denoted by
$$
\vartheta_\Y = F_0(\Omega^1_{R_\Y/R_o})
$$
and is called the {\it K\"{a}hler different} of~$R_\Y/R_o$
(or for $\Y$ w.r.t. $\{x_0,y_0\}$).
\end{definition}

Since the K\"ahler different is a bihomogeneous ideal
in~$R_{\Y}$, we can examine its bigraded Hilbert function.
The following proposition provides some basic properties
of this function.

\begin{proposition} \label{propSec5.12}
Let $\Y \!=\!\sum_{(i,j)\in D_\X} \!m_{ij}P_{ij}$
be an ACM fat point scheme in~$\popo$
supported at $\X$.
\begin{enumerate}
\item[(a)]  For all $(i,j)\in\N^2$, we have
  $\HF_{\vartheta_\Y}(i,j)\le \HF_{\vartheta_\Y}(i+1,j)$
  and $\HF_{\vartheta_\Y}(i,j)\le\HF_{\vartheta_\Y}(i,j+1)$.

\item[(b)] We have $\HF_{\vartheta_\Y} = 0$
  if and only if $m_{ij}\ge 2$ for all $P_{ij}\in\X$.

\item[(c)] Let $s'$ be the number of points
  $P_{ij}\in\X$ such that $m_{ij}=1$. Then we have
  $\HF_{\vartheta_\Y}(i,j) = s'$ for all large
  enough $i,j\gg 0$.
\end{enumerate}
\end{proposition}

\begin{proof}
Claim (a) follows from the fact that $\vartheta_\Y$
is a bihomogeneous ideal of~$R_\Y$ and $x_0,y_0$ are
non-zerodivisors of~$R_\Y$.

To show~(b), suppose that $m_{ij}\ge 2$ for all
$(i,j)\in D_\X$. Write
$\wp_{ij} = \langle L_{Q_i}, L_{R_j}\rangle$
with two linear forms $L_{Q_i}\in S_{1,0}$
and $L_{R_j}\in S_{0,1}$. Then
$I_\Y = \bigcap_{(i,j)\in D_\X}\wp_{ij}^{m_{ij}}$.
For any two bihomogeneous minimal generators $F,G$
of~$I_\Y$, it is easy to see that
$\frac{\partial(F,G)}{\partial(X_1,Y_1)} \in
\wp_{ij}^{m_{ij}}$ for all $(i,j)\in D_\X$.
Since $\vartheta_\Y$ is generated by the images
in~$R_\Y$ of the elements of the form
$\frac{\partial(F,G)}{\partial(X_1,Y_1)}$ where $F,G$
are bihomogeneous minimal generators of $I_\Y$, we get
$\vartheta_\Y =\langle 0\rangle$.

Conversely, w.l.o.g. assume that $m_{i_0j_0}=1$.
Let $F\in S$ be a separator for $P_{i_0j_0}$.
Obviously, we have $L_{Q_{i_0}}F,L_{R_{j_0}}F\in I_\Y$.
We write $P_{i_0j_0} =[a_0:a_1]\times[b_0:b_1]$
with $a_0\ne 0$ and $b_0\ne 0$,
and so $L_{Q_{i_0}}= a_1X_0-a_0X_1$ and $L_{R_{j_0}}=b_1Y_0-b_0Y_1$.
We also see that
$$
\begin{aligned}
\frac{\partial(L_{Q_{i_0}} F,L_{R_{j_0}}F)}{\partial(X_1,Y_1)}
&= \det
\begin{bmatrix}
F\tfrac{\partial L_{Q_{i_0}}}{\partial X_1}+
L_{Q_{i_0}}\tfrac{\partial F}{\partial X_1} &
F\tfrac{\partial L_{R_{j_0}}}{\partial X_1}+
L_{R_{j_0}}\tfrac{\partial F}{\partial X_1} \\
F\tfrac{\partial L_{Q_{i_0}}}{\partial Y_1}+
L_{Q_{i_0}}\tfrac{\partial F}{\partial Y_1} &
F\tfrac{\partial L_{R_{j_0}}}{\partial Y_1}+
L_{R_{j_0}}\tfrac{\partial F}{\partial Y_1}
\end{bmatrix}\\
&= F^2(\tfrac{\partial L_{Q_{i_0}}}{\partial X_1}
\tfrac{\partial L_{R_{j_0}}}{\partial Y_1} -
\tfrac{\partial L_{Q_{i_0}}}{\partial Y_1}
\tfrac{\partial L_{R_{j_0}}}{\partial X_1}) + G \\
&= a_0b_0F^2 + G
\end{aligned}
$$
for some $G\in I_\Y$. Since $a_0b_0\ne 0$, we obtain
$\overline{F}^2 \in \vartheta_\Y$. Furthermore, $F^2$ is also
a separator for $P_{i_0j_0}$, and hence $\overline{F}^2\ne 0$.

Finally we prove~(c).
Let $I = \bigcap_{(i,j)\in D_\X,m_{ij}>1} \wp_{ij}^{m_{ij}}$.
As in the proof of~(b), we have
$\vartheta_\Y \subseteq \wp_{ij}^{m_{ij}}/I_\Y$ if $m_{ij}>1$,
and so $\vartheta_\Y \subseteq I/I_\Y$. This implies
$$
\HF_{\vartheta_\Y}(i,j) \le \HF_{I/I_\Y}(i,j)
= \HF_\Y(i,j)-\HF_{S/I}(i,j)
$$
for all $(i,j)\in\N^2$.
Let $r= \#\pi_1(\X)$ and $t = \#\pi_2(\X)$.
Set $l = \sum_{i=1}^r \max\{m_{ie} \mid (i,e)\in D_\X\}$
and $l'= \sum_{j=1}^t \max\{m_{ej} \mid (e,j)\in D_\X\}$
(see also Notation~\ref{notationSec4.3}).
We have $\HF_\Y(i,j)=\sum_{(i,j)\in D_\X}\binom{m_{ij}+1}{2}$
and $\HF_{S/I}(i,j)=
\sum_{(i,j)\in D_\X, m_{ij}>1}\binom{m_{ij}+1}{2}$
for all $(i,j) \succeq(l,l')$, and so we get
$\HF_{\vartheta_\Y}(i,j) \le s'$ for $(i,j) \succeq(l,l')$.
It follows from (a) that $\HF_{\vartheta_\Y}(i,j) \le s'$
for all $(i,j)\in \N^2$.
Moreover, for each point $P_{ij}$ with $m_{ij}=1$ and
for which $F_{ij}$ is a minimal separator for $P_{ij}$,
we have $\overline{F}^2_{ij} \in \vartheta_\Y$
by the same reasoning as in the proof of part~(b).
According to Theorem~\ref{degP1xP1}, we know that
$
\deg(F_{ij}) =
\big( {\textstyle\sum\limits_{(e,j)\in D_\X}} m_{ej}-1,
{\textstyle\sum\limits_{(i,e)\in D_\X}} m_{ie}-1 \big).
$
We set
\[
t_1 = 2\max\big\{\,
{\textstyle\sum\limits_{(e,j)\in D_\X}} m_{ej}-1
\,\mid\, (i,j)\in D_\X, m_{ij}=1 \,\big\}
\]
and
\[
t_2 = 2\max\big\{\,
{\textstyle\sum\limits_{(i,e)\in D_\X}} m_{ie}-1
\,\mid\, (i,j)\in D_\X, m_{ij}=1 \,\big\}.
\]
Since $F^2_{ij}$ is also a separator for $P_{ij}$,
we have
$$
s' = \dim_K (\langle \overline{F}^2_{ij}
\,\mid\, (i,j)\in D_\X, m_{ij}=1\rangle)_{i',j'}
\le \HF_{\vartheta_\Y}(i',j') \le s'
$$
for all $(i',j') \succeq (t_1,t_2)$.
Hence the equality $\HF_{\vartheta_\Y}(i,j) = s'$ holds true
for all $(i,j) \succeq (t_1,t_2)$.
\end{proof}

If all multiplicities are 1, i.e., in the case
of an ACM set of reduced points, we can improve
the description of the Hilbert function
of~$\vartheta_{\X}$ as follows.

\begin{proposition}\label{HFofThetaForPoints}
Let $\X$ be an ACM set of reduced points in $\popo$,
let $r = \# \pi_1(\X)$, and let $t = \#\pi_2(\X)$.
\begin{enumerate}
\item[(a)] If $r \ge 2$, then $\HF_{\vartheta_\X}(0,j) = 0$
  for $j<t-1$.

\item[(b)] If $t \ge 2$, then $\HF_{\vartheta_\X}(i,0) = 0$
  for $i<r-1$.

\item[(c)] There exists a non-zerodivisor of $R_\X$
  which is contained in $\vartheta_\X.$

\item[(d)] Let $j\in\N$, let $\{h_1,\dots,h_u\}$ be a
  bihomogeneous minimal system of generators of~$\vartheta_\X$,
  let $\deg(h_k)=(i_k,j_k)$ for $k=1,\dots,u$,
  and let $i_0 = \max\{i_k \mid j_k\le j, k=1,\dots,u\}$.
  For $i\ge i_0$, if
  $\HF_{\vartheta_\X}(i,j)= \HF_{\vartheta_\X}(i+1,j)$
  then $\HF_{\vartheta_\X}(i+1,j)= \HF_{\vartheta_\X}(i+2,j)$.

\item[(e)] We have $\HF_{\vartheta_\X}(i,j) = s$ for
  all $(i,j) \succeq (2r-2,2t-2)$.
\end{enumerate}
\end{proposition}

\begin{proof}
First we prove claim~(a).
Then claim~(b) follows similarly.
Note that $\pi_2(\X) =\{R_1,\dots,R_{t}\}\subseteq \P^1$
is a complete intersection and
$I_{\pi_2(\X)} = \langle F\rangle\subseteq K[Y_0,Y_1]$
where $F=L_{R_1}\cdots L_{R_t}$ and $L_{R_j}$ is
the $(0,1)$-form that vanishes at $R_j$.
Let $\alpha_\X=(\alpha_1,\dots,\alpha_r)$ be the
associated tuple of $\X$.
Since $\X$ is ACM, we have $\alpha_1=t$ or
$\X_{Q_1}:=\pi_1^{-1}(Q_1)
=\{Q_1\times R_1,\dots,Q_1\times R_{t}\}$.
Then $I_{\X_{Q_1}} = \langle L_{Q_1}, F\rangle$.
Set $\V =\X\setminus\X_{Q_1}$. Clearly,
$F\in I_\X=I_\V \cap I_{\X_{Q_1}}.$

We want to show that $I_\X = L_{Q_1}I_\V+\langle F\rangle$.
It is clear that
$L_{Q_1}I_\V+\langle F\rangle \subseteq I_{\X}$.
Let $G\in I_{\X}.$ Then $G=L_{Q_1}H_1+FH_2$
with $H_1,H_2\in S$. Since $G,F\in I_{\X}$, we have
$L_{Q_1}H_1\in I_\X \subseteq I_{\V}$,
and so $(L_{Q_1}H_1)(Q\times R) =0$ for all $Q\times R\in\V$.
But $Q \ne Q_1$ for every $Q\times R\in\V$, this implies
$H_1(Q\times R)=0$ for every $Q\times R\in\V$.
Thus $H_1\in I_{\V}$.

If $r=2$, then $I_{\V}=I_{\X_{Q_2}}=\langle L_{Q_2}, G\rangle$
for some $G\in K[Y_0,Y_1]$.
Then $I_{\X}=\langle L_{Q_1}L_{Q_2},L_{Q_1}G,F \rangle$.
The element $\frac{\partial(L_{Q_1}G,F)}{\partial(x_1,y_1)}$
has degree $(0, t-1)+\deg(G)=(0,j_0)$
and any other bihomogeneous element of~$\vartheta_\X$
has degree $(0,j)$ with $j\ge j_0>t-1$.
Hence we get $\HF_{\vartheta_\X} (0,j)=0$ for $j\le t-1$.
In the case $r>2$ we also have the above equality, since
$I_{\V}\subseteq I_{\X_{Q_2}}=\langle L_{Q_2}, G\rangle$.

Now we prove (c). Since $\pi_1(\X)$ and $\pi_2(\X)$
are complete intersections in $\P^1$, we write
$I_{\pi_1(\X)} = \langle F_1 \rangle \subseteq K[X_0,X_1]$
and $I_{\pi_2(\X)} = \langle F_2\rangle\subseteq K[Y_0,Y_1]$,
where $F_1$ and $F_2$ are bihomogeneous.
Then $\frac{\partial(F_1,F_2)}{\partial(x_1,y_1)} =
\frac{\partial F_1}{\partial x_1}\cdot
\frac{\partial F_2}{\partial y_1} \in \vartheta_\X$.
Note that $\frac{\partial F_1}{\partial x_1}(Q_i) \ne 0$ for
all $Q_i\in \pi_1(\X)$ and
$\frac{\partial F_2}{\partial y_1}(R_j) \ne 0$ for
all $R_j\in \pi_2(\X)$. Thus the element
$\frac{\partial(F_1,F_2)}{\partial(x_1,y_1)}$
is a non-zerodivisor of~$R_\X$.

For proving~(d), suppose that
$\HF_{\vartheta_\X}(i,j)= \HF_{\vartheta_\X}(i+1,j)$.
Because $x_0$ is a non-zerodivisor of $R_\X$,
the multiplication map
$(\vartheta_\X)_{i,j}
\stackrel{\mu_{x_0}}{\longrightarrow}
(\vartheta_\X)_{i+1,j}$
is an isomorphism of $K$-vector spaces.
So, we have
$(\vartheta_\X)_{i+1,j}=x_0\cdot(\vartheta_\X)_{i,j}$.
Obviously, $x_0\cdot(\vartheta_\X)_{i+1,j}\subseteq
(\vartheta_\X)_{i+2,j}$.
For the other inclusion, let
$f \in (\vartheta_\X)_{i+2,j}\setminus\{0\}$.
Since $i\ge i_0$, we may write $f=x_0f_0+x_1f_1$
where $f_0,f_1\in (\vartheta_\X)_{i+1,j}$.
We write $f_k=x_0g_k \in x_0\cdot(\vartheta_\X)_{i,j}$
for $k=0,1$. This implies
$$
f=x_0f_0+ x_1f_1
=x_0(x_0g_0+x_1g_1)\in x_0\cdot(\vartheta_\X)_{i+1,j}.
$$
Hence $x_0\cdot(\vartheta_\X)_{i+1,j}=
(\vartheta_\X)_{i+2,j}$, and consequently
$\HF_{\vartheta_\X}(i+1,j)= \HF_{\vartheta_\X}(i+2,j)$.

Finally, claim (e) follows from the preceding proposition.
\end{proof}

\begin{example}
Let $i\ge 0$, let $Q_i = R_i = [1:i] \in \P^1$, and let
$P_{ij}$ denote the point $Q_i\times R_j$ in~$\popo$.
We let $\X$ be the set of points
$\X =\{P_{11},P_{12},P_{13},P_{21},P_{22}\}$
in~$\popo$.
Then $\X$ is ACM, we have $r=2$ and $t=3$,
and the Hilbert function of $\vartheta_\X$ is
$$
\HF_{\vartheta_\X} =
\left[ \begin{smallmatrix}
  0 & 0 & 0 & 0 & 1 & 1 & \dots\\
  0 & 0 & 1 & 2 & 3 & 3 & \dots\\
  0 & 1 & 3 & 4 & 5 & 5 & \dots\\
  0 & 1 & 3 & 4 & 5 & 5 & \dots\\
  0 & 1 & 3 & 4 & 5 & 5 & \dots\\
  0 & 1 & 3 & 4 & 5 & 5 & \dots \\
\svdots & \svdots & \svdots &\svdots &
\svdots &\svdots &  \sddots
\end{smallmatrix}\right].
$$
In this case we see that $\HF_{\vartheta_\X}(i,j)=5$
for all $(i,j) \succeq (2r-2,2t-2)= (2,4)$.
\end{example}

%
%

\bigbreak
\section{The Cayley-Bacharach Property}

In this section we consider sets of reduced points in~$\popo$
having the Cayley-Bacharach property.

First we recall this property for a set of reduced points
$\X=\{P_1,\dots,P_s\}$ in~$\P^n$. Let $r_\X$ be the regularity
index of $\X$, i.e., the least degree such that the Hilbert
function of $\X$ equals the degree of~$\X$.
The set $\X$ is called a {\it Cayley-Bacharach scheme}
if every hypersurface of degree $r_\X-1$ which contains all
but one point of $\X$ must contain all points of~$\X$.

Notice that a set of distinct points $\X\subseteq \P^n$
is a Cayley-Bacharach scheme if and only if
$\HF_{\X\setminus \{P_{j}\}}$ does not depend on
the choice of $j$. Moreover, one can detect a Cayley-Bacharach
scheme (especially, a Cayley-Bacharach scheme
being a complete intersection) by looking at
a particular homogeneous component of its K\"{a}hler different
(see \cite[Lemma~3.7]{KLL15} and \cite[Thm.~5.6]{KL16}).

In Section \ref{CI}, Propositions~\ref{indip1} and~\ref{indip2}
we showed that the Hilbert function of the K\"{a}hler
differential module for subschemes obtained by reducing
by one the multiplicity of one point $P_{ij}$
in an equimultiple fat point scheme in~$\popo$ supported
on a complete intersection does not depend on the choice of~$(i,j)$.
For reduced schemes, this leads us to examine the
Cayley-Bacharach property in~$\popo$ which is defined
as follows.

\begin{definition}
Let $\X$ be a set of distinct points in $\popo$.
We say that $\X$ has the {\it Cayley-Bacharach property}
if the Hilbert function of $\X \setminus \{P_{ij}\}$
is independent of the choice of $P_{ij}\in \X$.
\end{definition}

\begin{remark}\label{remS5.3}
For any finite set $\Sigma\subseteq \mathbb{N}^2$ we set
$$
D_{\Sigma} = {\textstyle\bigcup\limits_{(i,j)\in\Sigma}}
\big\{\, (k,l)\in \mathbb{N}^2 \,\mid\, (k,l)\succeq(i,j)\,\big\}.
$$
If $\X$ is a set of distinct points in $\popo$ and $P_{ij}\in\X$,
\cite[Thm.~2.2]{GV08} shows that
$$
\HF_{\X\setminus\{P_{ij}\}}(i,j) =
\begin{cases}
\HF_{\X}(i,j) &\  \mbox{if $(i,j)\notin D_{\deg_\X(P_{ij})}$},\\
\HF_{\X}(i,j)-1 &\  \mbox{if $(i,j)\in D_{\deg_\X(P_{ij})}$}.\\
\end{cases}
$$
Thus we can say that $\X$ has the Cayley-Bacharach property
if and only if all of its points have the same minimal
separator degree.
\end{remark}

From Theorem~\ref{degP1xP1} we see that an ACM set of points
$\X\subseteq \popo$ with the associated tuples
$\alpha_\X =(\alpha_1,\dots,\alpha_r)$ and
$\beta_\X =(\beta_1,\dots,\beta_t)$ satisfies
$$
\deg_\X(P_{ij}) = (\beta_j-1,\alpha_i-1)
$$
for all $P_{ij}\in\X$. So, Remark~\ref{remS5.3} yields that $\X$
has the Cayley-Bacharach property if and only if $\alpha_1 =
\dots=\alpha_{r}=t$ and $\beta_1=\cdots=\beta_{t}=r$. By
Theorem~\ref{pnmsepfrombetti} (or~\cite[Thm.~4.1]{GV08}),
this is also equivalent to the fact that~$\X$ is a complete
intersection.

\begin{proposition} \label{CBCI}
Let $\X$ be an ACM set of reduced points in $\popo$,
and let $r = \#\pi_1(\X)$ and $t=\#\pi_2(\X)$.
The following statements are equivalent:
\begin{enumerate}
\item[(a)] $\X$ has the Cayley-Bacharach property.
  
\item[(b)] $\X$ is a complete intersection.
  
\item[(c)]
  For every point $P_{ij}\in\X$, the K\"{a}hler different
  $\vartheta_\X$ contains no separator for $P_{ij}$ of degree
  $\prec~(2r-2,2t-2)$.
\end{enumerate}
If one of these conditions is satisfied, the Hilbert function
of $\vartheta_\X$ is given by
$$
\HF_{\vartheta_\X}(i,j) = \HF_\X(i-r+1, j-t+1)
$$
for all $(i,j)\in\N^2$.
\end{proposition}

\begin{proof}
The equivalence of (a) and (b) follows from the above argument.
So, it suffices to show that these claims are equivalent to~(c).
Suppose that $\X$ is a $CI(d_1,d_2)$.
We have $d_1 = r$, $d_2 =t$ and $I_\X = \langle F_1,F_2\rangle$
with $F_1 \in S_{d_1,0}$ and $F_2 \in S_{0,d_2}$.
Then we get $\vartheta_\X = \langle \frac{\partial(F_1,F_2)}
{\partial(x_1,y_1)} \rangle$. As in the proof
Proposition~\ref{HFofThetaForPoints}.c,
the element $\frac{\partial(F_1,F_2)}{\partial(x_1,y_1)}$ is
a non-zerodivisor of~$R_\X$.
Moreover, every point $P_{ij}\in\X$ has degree
$\deg_\X(P_{ij})=(r-1,t-1)$. If $\vartheta_\X$ contains
a separator $\overline{F}_{ij}$ for some $P_{ij}$
of degree $\prec~(2r-2,2t-2)$.
Then $F_{ij} = \frac{\partial(F_1,F_2)}{\partial(X_1,X_1)}
\cdot H_{ij}$ with $\deg(H_{ij}) \prec~(r-1,t-1)$.
But in this case $H_{ij}$ is also a separator for $P_{ij}$,
and so $\deg_\X(P_{ij})\prec (r-1,t-1)$, a contradiction.

Conversely, suppose that $\vartheta_\X$ contains no separator
for $P_{ij}$ of degree $\prec~(2r-2,2t-2)$ for all $P_{ij}\in\X$
and that $\X$ does not have the Cayley-Bacharach property.
Then there is a point $P_{ij}\in\X$
such that $\deg_\X(P_{ij})\prec (r-1,t-1)$. Let $F_{ij}$ be a
minimal separator for $P_{ij}$ of degree $(i,j)\prec (r-1,t-1)$.
As in the proof of Proposition~\ref{propSec5.12}, we have
$\overline{F}^2_{ij} \in \vartheta_\X$ and
$\deg(\overline{F}^2_{ij}) = (2i,2j)\prec~(2r-2,2t-2)$.
Furthermore, $\overline{F}^2_{ij}$ is also a separator
for $P_{ij}$.
This is a contradiction.
\end{proof}

As an immediate consequence of the previous proposition
and~\cite[Thm.~4.2]{GV04}, the first difference
function of the Hilbert function of the K\"{a}hler different
of a complete intersection $\X$ in~$\popo$ can be
described explicitly as in our next corollary.

\begin{corollary}
Suppose $\X$ is a $CI(d_1,d_2)$ in~$\popo$
with $d_1\leq d_2$. Then the difference function of
$\HF_{\vartheta_\X}$ is given by
\begin{small}
$$
\left[
\begin{array}{rrrrr rrrrr rrrrr rrrrr rrr}

     {d_{1}\!-\!1} \left\{
     \begin{array}{cccccccc}    
       0      & \dots  & \dots &0  \\
       \vdots & \ddots &\ddots &    \\
       0      & \dots  & \dots &0   \\
     \end{array}\right.
       \begin{array}{cccccc}     
        0      & 0    &\dots  & 0 \\
               &\ddots&\ddots &   \\
        0      & 0    &\dots  & 0  \\
       \end{array}
         \begin{array}{cccccccc}  
          0 &\dots&\dots&0 &\dots\\
            & \ddots & \ddots &  &    \\
          0 &\dots&\dots& 0 &\dots \\
         \end{array}
    \\
     \underbrace{
     d_1\,\left\{
     \begin{array}{cccccccc}     
       0      & \dots  &\dots  &0  \\
       \vdots & \ddots &\ddots &    \\
       0      & \dots  &\dots  &0   \\
     \end{array}\right.
      }_{d_2-1}
    \underbrace{
       \begin{array}{cccccc}     
       1     & 1    & \dots &  1 \\
             &\ddots& \ddots &    \\
       1     &1     & \dots & 1  \\
       \end{array}
       }_{d_2}
         \begin{array}{cccccccc}  
          0 &\dots&\dots&0 &\dots\\
            & \ddots & \ddots &  &    \\
          0 &\dots&\dots& 0 &\dots \\
         \end{array}
    \\
     \left.
     \begin{array}{cccccccc}     
       0      & \dots  &\dots  &0  \\
       \vdots &        &       &\vdots
     \end{array}
     \right.
       \begin{array}{cccccc}     
        \, 0  \,  & \, 0\,    & \,\dots & 0 \\
        \vdots &\vdots&      &\vdots
       \end{array}
         \begin{array}{cccccccc} 
          0 & \dots&\dots& 0 & \dots   \\
          \vdots& & &\vdots&\ddots
         \end{array}
\end{array}\right ] .
$$
\end{small}
\end{corollary}

\newpage
%
%

\end{document}